\documentclass{article}
\usepackage{a4wide}
\usepackage[utf8]{inputenc}
\usepackage{stmaryrd,mathrsfs,bm,amsthm,mathtools,yfonts,amssymb,esint}
\usepackage{color}
\numberwithin{equation}{section}

\newtheorem{Theorem}{Theorem}[section]
\newtheorem{Proposition}{Proposition}[section]
\newtheorem{Lemma}{Lemma}[section]
\newtheorem{Corollary}{Corollary}[section]
\newtheorem{Definition}{Definition}[section]

\newcommand{\bTheorem}[1]{
\begin{Theorem} \label{T#1} }
\newcommand{\eT}{\end{Theorem}}

\newcommand{\bProposition}[1]{
\begin{Proposition} \label{P#1}}
\newcommand{\eP}{\end{Proposition}}

\newcommand{\bLemma}[1]{
\begin{Lemma} \label{L#1} }
\newcommand{\eL}{\end{Lemma}}

\newcommand{\bCorollary}[1]{
\begin{Corollary} \label{C#1} }
\newcommand{\eC}{\end{Corollary}}

\newcommand{\bFormula}[1]{
\begin{equation} \label{#1}}
\newcommand{\eF}{\end{equation}}
\newcommand{\vc}[1]{{\bf #1}}
\newcommand{\Ov}[1]{\overline{#1}}

\newcommand{\vU}{\vc{U}}
\newcommand{\vu}{\vc{u}}
\newcommand{\vr}{\varrho}
\newcommand{\vre}{\vr_\ep}
\newcommand{\vte}{\vt_\ep}
\newcommand{\vue}{\vu_\ep}

\newcommand{\vt}{\vartheta}
\newcommand{\vn}{\vc{n}}
\newcommand{\vV}{\vc{V}}
\newcommand{\vrb}{\Ov{\vr}}
\newcommand{\vtb}{\Ov{\vt}}
\renewcommand{\H}{\vc{H}}
\newcommand{\Ht}{\vc{H}_t}
\newcommand{\HH}{\mathbb{H}}

\newcommand{\Hpt}{\H^{\perp}_t}
\newcommand{\vz}{\vc{z}}
\newcommand{\vze}{\vz_\ep}
\newcommand{\vh}{\vc{h}}
\newcommand{\va}{\vc{a}}
\newcommand{\vQ}{\vc{Q}}

\newcommand{\om}{\omega}

\newcommand{\Div}{{\rm div}_x}
\newcommand{\Grad}{\nabla_x}

\newcommand{\tn}[1]{\mbox {\F #1}}
\newcommand{\dx}{\,{\rm d} {x}}
\newcommand{\dy}{\,{\rm d} {y}}
\newcommand{\dS}{\,{\rm d} {S}}
\newcommand{\vph}{{\boldsymbol \varphi}}
\newcommand{\vq}{\mathbf{q}}
\newcommand{\dt}{\,{\rm d} t }

\newcommand{\dxdt}{\,{\rm d}x{\rm d}t}

\newcommand{\intint}[1]{\int_0^T \int_{\Omega_t} #1 \ \dxdt}

\newcommand{\vv}{\vc{v}}

\newcommand{\ep}{\varepsilon}

\font\F=msbm10 scaled 1000
\newcommand{\R}{\mbox{\F R}}

\newcommand{\Rtri}{\mathbb{R}^3}
\newcommand{\de}{\partial}

\providecommand{\abs}[1]{\left\lvert#1\right\rvert}
\newcommand{\sil}{\rightarrow}

\newcommand{\weak}{\rightharpoonup}
\newcommand{\weakstar}{\stackrel{\ast}{\rightharpoonup}} 

\newcommand{\del}{d_{\ep,l}}
\newcommand{\psel}{\Psi_{\ep,l}}

\newcommand{\felj}{f_{\ep,l}^1}
\newcommand{\feld}{\vc{f}_{\ep,l}^2}

\definecolor{grey}{rgb}{0.85,0.85,0.85}

\long\def\greybox#1{%
    \newbox\contentbox%
    \newbox\bkgdbox%
    \setbox\contentbox\hbox to \hsize{%
        \vtop{
            \kern\columnsep
            \hbox to \hsize{%
                \kern\columnsep%
                \advance\hsize by -2\columnsep%
                \setlength{\textwidth}{\hsize}%
                \vbox{
                    \parskip=\baselineskip
                    \parindent=0bp
                    #1
                }%
                \kern\columnsep%
            }%
            \kern\columnsep%
        }%
    }%
    \setbox\bkgdbox\vbox{
        \color{grey}
        \hrule width  \wd\contentbox %
               height \ht\contentbox %
               depth  \dp\contentbox
        \color{black}
    }%
    \wd\bkgdbox=0bp%
    \vbox{\hbox to \hsize{\box\bkgdbox\box\contentbox}}%
    \vskip\baselineskip%
}

\title{Low stratification of a heat-conducting fluid in time-dependent domain}

\author{
Ond{\v r}ej Kreml $^{1}$\thanks{The works of O.K., V.M. and   \v S.N. were supported by  project GA16-03230S and by RVO 67985840.  Part of work was done during  stay of O.K. at Imperial College London which  was supported by the grant Iuventus Plus 0871/IP3/2016/74.}
\and V\'aclav M\'acha $^{1}$\footnotemark[1]
\and {\v S}{\' a}rka Ne{\v c}asov{\' a} $^1$\footnotemark[1]
\and Aneta Wr\'oblewska-Kami\'nska $^2$\thanks{The work of A.W.-K. 
is partially supported by a Newton Fellowship of the
Royal Society and by the grant Iuventus Plus 0871/IP3/2016/74 of Ministry of Sciences and Higher
Education RP. Her stay at Institute of Mathematics of Academy of Sciences, Prague was supported by 7AMB16PL060.}
}

\begin{document}
\maketitle

\bigskip

\centerline{$^1$Institute of Mathematics of the Academy of Sciences of
the Czech Republic} \centerline{\v Zitn\' a 25, 115 67 Praha 1,
Czech Republic}
\bigskip

\centerline{$^2$  Institute of Mathematics, Polish Academy of Sciences}
\centerline{\'Sniadeckich 8, 00-656 Warszawa, Poland}

\medskip
\abstract{
We study the low Mach number limit of the full Navier-Stokes-Fourier system in the case of low stratification with ill-prepared initial data for the problem stated on moving domain with prescribed motion of the boundary. Similarly as in the case of a fixed domain we recover as a limit the Oberback-Boussinesq system, however we identify one additional term in the temperature equation of the limit system which is related to the motion of the domain and which is not present in the case of a fixed domain. One of the main ingredients in the proof are the properties of the Helmholtz decomposition on moving domains and the dependence of eigenvalues and eigenspaces of the Neumann Laplace operator on time.






}

\section{Introduction}\label{s:intro}

The mathematical theory of singular limits of systems of equations describing fluid motion goes back to the seminal work of Klainerman and Majda \cite{KM1}. The motivation for a study of such type of limits follows from the generality of corresponding equations. More precisely, in a case of the full Navier-Stokes-Fourier system the equations describe a spectrum of possible motions e.g. sound waves or models of gaseous stars in astrophysics. Such type of study allows us to eliminate unimportant or unwanted modes of motion, as a consequence of scaling and asymptotic analysis. The aim of the asymptotic analysis of various physical systems is to derive a simplified system of equations which can be solved numerically or analytically see  e.g. Zeytounian \cite{Z}.

The goal of the mathematical analysis of low Mach number limits is to fill up the gap between the compressible fluids and their "idealized" incompressible models. There are two ways how to introduce the Mach number into the system, from the physical point of view different but from the mathematical point of view completely equivalent. The first approach considers a varying equation of state as well as the transport coefficients see work of Ebin \cite{EB1},  Schochet \cite{SCH2}. 
The second way is to evaluate qualitatively the incompressibility using the dimensional analysis.  We rewrite our system in the dimensionless form by scaling each variable by its characteristic value, see Klein \cite{Kl}. 

The mathematical analysis of singular limits in the frame of strong solutions can be referred to works of Gallager \cite{Gallag}, Schochet \cite {SCH2}, Danchin \cite {Da}, Hoff \cite{Ho}. The seminal works of Lions \cite{LI4} and the extension by Feireisl et al. \cite{FNP} on the existence of global weak solutions in the barotropic case gave a new possibility of rigorous study of the singular limits in the frame of weak solutions see work of Desjardins and Grenier \cite{DesGre}, Desjardinds, Grenier, Lions and Masmoudi \cite{DGLM}. 

The mathematical theory of the full Navier-Stokes-Fourier system was studied by Feireisl. First he developed a concept of variational solution under the assumption that the pressure can be decomposed into the elastic part and the thermal part. He introduced a definition of the weak solution using the thermal energy inequality instead of the energy equation and  complementing the system with the global total energy inequality \cite{EF70,EF71}. Later, he developed a new concept of a variational solution to the compressible Navier-Stokes-Fourier
system based on the principles of  global energy conservation and
 nonnegative local entropy production. This concept in particular allowed to develop important new results concerning the singular limits of variational solutions to the full system, see Feireisl, Novotn\' y \cite{FeNo6}.  Let us also mention that for the full system a low Mach number convergence on a short time interval within the framework of regular solutions was proved by Alazard \cite{A}.

In real world applications there are many problems where the fluid interacts with a boundary of its container which is not fixed and moves either by a prescribed motion or the motion of the boundary is related to the motion of the fluid. The mathematical theory of such motions then becomes even more complex and additional difficulties arise. In this paper we study the first, and arguably the easier case, namely a problem in a moving domain whose motion is prescribed by a given velocity field $\vV(t,x)$.

In the barotropic case, the existence theory of global weak solution was proved by Feireisl et al. \cite{FeNeSt,FeKrNeNeSt} for the Dirichlet and Navier type of boundary conditions, respectively. Moreover, in the framework of weak solutions the singular limit (low Mach number limit) in the case of moving domain was  investigated by Feireisl et al. in \cite{FKNNS14,FeKrMaNe}.

Concerning the full Navier-Stokes-Fourier system, the global existence of weak/variational solutions was extended to the case of moving domain by Kreml et al. see \cite{KMNW2,KMNW1}.

The aim of this paper is to fill up the gap of theory of singular limits by examining the low Mach number limit for the full Navier-Stokes-Fourier system on moving domains. We consider a low stratification with ill-prepared initial data. For a fixed domain the target system is the Oberback-Boussinesq system and the convergence of sequence of variational solutions to the primitive system to the weak solution of the target system was proved by Feireisl and Novotn\'y \cite{FeNo6}. Since the domain in our case is moving we can no longer assume that the potential of the driving force $F(x)$ satisfies $\int_{\Omega_t} F(x) \dx =0$. This is the reason why in the limit we recover the Oberback-Boussinesq system with a new additional term. 

The paper is structured as follows. In Section \ref{s:intro}, the variational formulation of the {\it primitive system} (scalled system) is introduced and the existence theorem is stated. 
Section \ref{s:target} is devoted to the {\it target system} - limit system where the Oberback-Boussinesq system is recovered as a low Mach number limit of the full system which in particular differs from case of a fixed domain. 
In Section \ref{s:lowmach} we state the uniform estimates and perform the limits in the continuity equation, entropy balance and momentum equation where the limit of the convective term remains unspecified. 

Finally, Section \ref{s:convec} is devoted to the study of the limit of the convective term. Firstly, we mention that in comparison with fixed domain the Helmholtz decomposition depends on time. The main problem is a possible development of fast oscillations in the momenta $ \vre\vue$, $\ep \to 0$ with respect to time.  We show the compactness of the solenoidal part of velocity field similarly as in the case of fixed domain. To prove the convergence of the gradient part it is necessary to introduce the acoustic equations. This idea goes back to Schochet who  found that singular component of the gradient part of the momentum together with a certain function of $\vr_{\epsilon}, \vt_{\epsilon}$
satisfy a linear wave equation.
We present the reduction to finite number of modes and as in the barotropic case on moving domain we deal with the fact that the eigenvalues and eigenfunctions of the Neumann Laplace equation depend on time.


\subsection{Primitive system}
Let us consider the full system on time dependent domain in low Mach number regime which is given by the following
\bFormula{i1a}
\partial_t \vr + \Div (\vr \vu) = 0,
\eF
\bFormula{i1b}
\partial_t (\vr \vu) + \Div (\vr \vu \otimes \vu) + \frac{1}{\ep^2}\Grad p = \Div \tn{S} + \frac{1}{\ep}\vr \nabla_x F,
\eF
\bFormula{i1c}
\partial_t (\vr s) + \Div (\vr s \vu) + \Div \left(\frac{\vq}{\vt}\right) = \sigma_\ep,
\eF
\bFormula{i1d}
\frac{{\rm d}}{{\rm d}t} \int \left(\frac{\ep^2}{2} \vr |\vu|^2 + \vr e -\ep\vr F \right)\, \dx = 0.
\eF
The entropy production measure $\sigma_\ep$ satisfies 
	\bFormula{sigma1}
	\sigma_\ep \geq \frac{1}{\vt} \left( \ep^2 \tn{S} : \Grad \vu  -\frac{\vq}{\vt} \cdot \Grad \vt \right).
	\eF
 In particular the number $\ep >  0$ is related to the choice of Mach number ($=\frac{\vu_{char}}{\sqrt{p_{char}/\varrho_{char}}}$) to be sufficiently small (the speed of sound dominates characteristic fluid velocity) and a Froude number to be equal $\sqrt{\ep}$ which is related to the low stratification.   
We consider this system of equations 
 being mathematical formulations of the balance of mass, linear momentum, entropy and total energy respectively and to be satisfied on the space-time cylinder $Q_T =  \cup_{t\in(0,T)} \{t\}  \times \Omega_t $ describing a physical domain moving in time. 
Unknowns are the density $\vr:Q_T\mapsto [0,\infty)$, the velocity $\vu:Q_T\mapsto \mathbb R^3$ and the temperature $\vt:Q_T\mapsto [0,\infty)$. The potential of the external body force $F = F(x)$ is assumed to be independent of time. Other quantities appearing in these equations are functions of the unknowns, namely the stress tensor $\tn{S}$, the internal energy $e$, the pressure $p$, the entropy $s$, and the entropy production rate $\sigma_\ep$.

To be more precise, the time dependent domain $\Omega_t$ is prescribed by movement of its boundary on the time interval $[0,T]$. In order to describe this movement we consider a given velocity field $\vc{V}(t,x)$ for $t \geq 0$ and $x \in \mathbb{R}^3$ which is smooth enough. Then the position of the domain $\Omega_t$ at time $t > 0$ is given by the solution to the associated system of equations 
\begin{equation}
\frac{{\rm d}}{{\rm d}t} \vc{X}(t, x) = \vc{V} \Big( t, \vc{X}(t, x) \Big),\ t > 0,\ \vc{X}(0, x) = x\label{X.t},
\end{equation}
and by a given bounded initial domain $\Omega_0 \subset \mathbb{R}^3$ as
\begin{equation}
\Omega_\tau = \vc{X} \left( \tau, \Omega_0 \right),
\quad  \mbox{ with } \quad
\Gamma_\tau = \partial \Omega_\tau.
\end{equation}

The system of equations \eqref{i1a}-\eqref{i1d} is complemented by the following boundary conditions. We assume that the boundary of the domain is impermeable, hence
\bFormula{i6} (\vu - \vc{V})
\cdot \vc{n} |_{\Gamma_\tau} = 0 \ \mbox{for any}\ \tau \geq 0,
\eF
where $\vc{n}(t,x)$ denotes the unit outer normal vector to $\Gamma_t$. We prescribe full slip boundary condition for the velocity field $\vu$, meaning 
\bFormula{b1} \left[ \tn{S} \vc{n} \right]\times \vc{n} |_{\Gamma_\tau} = 0
\eF
and for the heat flux -- the conservative boundary condition
\bFormula{b2} \vq\cdot \vc{n} |_{\Gamma_\tau}=0 .
\eF
Additionally, we assume that the moving domain does not change its volume ($|\Omega_\tau| = |\Omega_0|$ for any $\tau \geq 0$). Namely, it is possible to choose $\vV$ such that
\begin{equation}\label{divV}
\Div \vc{V} (\tau, \cdot) = 0 \mbox{ for any }\tau \geq 0. 
\end{equation}

Finally, the system \eqref{i1a}-\eqref{i1d} is supplemented with initial conditions 
$
\varrho_{0},$ $\vc{u}_{0},$ $\vt_{0}$
and we denote
$e_0:= e(\vr_0,\vt_0)$ and $s_0 := s(\vr_0,\vt_0).$
In particular, we assume that the initial data are ill-prepared and take the form
\begin{equation}\label{initial_1}
\varrho_{0,\varepsilon} = \overline \varrho + \varepsilon\varrho^{(1)}_{0,\varepsilon},\quad 
\vartheta_{0,\varepsilon} = \overline \vartheta + \varepsilon \vartheta^{(1)}_{0,\varepsilon}, \quad \mbox{ where }\overline\varrho>0,\ \overline\vartheta >0 \mbox{ are positive constants, }
\end{equation}
\begin{equation}\label{initial_2}
\int_{\Omega_0} \varrho^{(1)}_{0,\varepsilon}\dx = 0, \quad \int_{\Omega_0} \vartheta^{(1)}_{0,\varepsilon}\dx = 0 \quad \mbox{ for all }\ep>0
\end{equation}
and 
\begin{equation}\label{initial_3}
\varrho^{(1)}_{0,\varepsilon}, \ \vu_{0,\ep}, \ \vartheta^{(1)}_{0,\varepsilon} \mbox{ are bounded measurable functions for all }\ep>0.
\end{equation}

\subsection{Hypotheses}
\label{s:m1}

Motivated by \cite{KMNW2,FeNo6} we introduce the following set of assumptions, which allow to obtain the existence of weak solutions.

The stress tensor $\tn{S}$ is determined by the standard Newton rheological law
\bFormula{i4}
\tn{S} (\vt, \Grad \vu) = \mu(\vt) \left( \Grad \vu +
\Grad^t \vu - \frac{2}{3} \Div \vu \tn{I} \right) + \eta (\vt) \Div \vu
\tn{I},\,\, \mu > 0,\ \eta \geq 0.
\eF
We assume the viscosity coefficients $\mu$ and $\eta$ are continuously differentiable functions of the absolute temperature, namely $\mu,\ \eta \in C^1[0,\infty)$ and satisfy
	\bFormula{mu_1}
	0< \underline{\mu} (1+ \vt) \leq \mu(\vt) \leq \overline{\mu} (1+\vt),
	\quad
	\sup_{\vt \in [0,\infty)} |\mu'(\vt)| \leq \overline{m},
	\eF
	\bFormula{eta_1}
	0 \leq \eta(\vt) \leq \overline{\eta} (1+\vt).
	\eF
The heat flux $\vq$ satisfies the Fourier law for in the  following form
 \bFormula {q1}
 \vq= -\kappa (\vt)\Grad\vt,
\eF
where the heat coefficient $\kappa$ can be decomposed into two parts
	\bFormula{kappa_1}
	\kappa(\vt) = \kappa_M (\vt) + \kappa_R (\vt), \quad \mbox{where }\kappa_M, \ \kappa_R \in C^1[0,\infty) ,
	\eF
	\bFormula{kappa_2}
	0< \underline{\kappa_R} (1+ \vt^3) \leq \kappa_R (\vt) \leq \overline{\kappa_R}(1+ \vt^3),
	\eF
	\bFormula{kappa_3}
	0< \underline{\kappa_M}   (1+ \vt) \leq \kappa_M (\vt) \leq \overline{\kappa_M}  (1+ \vt).
	\eF
In the above formulas $\underline{\mu}$, $\overline{\mu}$, $\overline{m}$, $\overline{\eta}$, $\underline{\kappa_R}$, $\overline{\kappa_R}$, $\underline{\kappa_M} $, $\overline{\kappa_M}$ are positive  constants.

The quantities $p$, $e$, and $s$ are continuously differentiable functions for positive values of $\vr$, $\vt$ and satisfy Gibbs' equation 
	\bFormula{gibs}
	\vt D s(\vr,\vt) = D e(\vr,\vt) + p(\vr,\vt) D\left( \frac{1}{\vr}  \right) \mbox{ for all } \vr, \ \vt > 0.
	\eF

Further, we assume the following state equation for the pressure and the internal energy
\bFormula{p1p}
p(\vr,\vt) = p_M(\vr,\vt)+ p_R(\vt), \quad p_R(\vt) = \frac a3 \vt^4,\ a>0,
\eF
\bFormula{e1e}
e(\vr,\vt) = e_M(\vr,\vt) + e_R(\vr,\vt), \quad \vr e_R(\vr,\vt) = a\vt^4,
\eF
and
\bFormula{s1s}
s(\vr,\vt) = s_M(\vr,\vt) + s_R(\vr,\vt), \quad \vr s_R(\vr,\vt) = \frac 43 a\vt^3.
\eF
According to the hypothesis of thermodynamic  stability the molecular components satisfy
	\bFormula{pm_1}
	\frac{\partial p_M}{\partial \vr} >0 \mbox{ for all } \vr, \ \vt >0
	\quad
\mbox{ and } 
	\quad
	0<\frac{\partial e_M}{\partial \vt} \leq c \mbox{ for all }  \vr, \ \vt >0.
	\eF
Moreover,
	\bFormula{em2}
	\lim_{\vt\to 0^+} e_M(\vr,\vt) = \underline{e}_M(\vr) > 0 \mbox{ for any fixed } \vr >0,
	\eF	
and 
	\bFormula{em3}
	\left| \vr \frac{\partial e_M(\vr,\vt)}{\partial \vr} \right| \leq c e_M(\vr, \vt) \mbox{ for all } \vr, \ \vt >0.
	\eF
We suppose that there is a function $P$ satisfying 
	\bFormula{p2p}
	P \in C^1[0,\infty), \ P(0)=0,\  P'(0)>0,
	\eF
and two positive constants $0< \underline{Z} < \overline{Z}$ such that
	\bFormula{pm2}
	p_M(\vr,\vt) =  \vt^{\frac{5}{2}} P\left( \frac{\vr}{\vt^{\frac{3}{2}}} \right) 
	\mbox{ whenever } 0< \vr \leq \underline{Z} \vt^{\frac{3}{2}},\mbox{ or, } \vr > \overline{Z} \vt^{\frac{3}{2}} 
	\eF
and 
	\bFormula{pm3}
	p_M(\vr,\vt) = \frac{2}{3} \vr e_M (\vr,\vt) \mbox{ for } \vr > \overline{Z}\vt^{\frac{3}{2}}.
	\eF


\subsection{Variational formulation of the primitive system}

We work with a variational formulation of the primitive system \eqref{i1a}-\eqref{i1d}. Namely, the equation \eqref{i1a} is fulfilled in the sense of renormalized solutions introduced by DiPerna and Lions \cite{DL}:
\bFormula{m2}
\int_0^T \int_{\Omega_t} \vr B(\vr) ( \partial_t \varphi + \vu \cdot \Grad \varphi )\, \dxdt = 
\int_0^T \int_{\Omega_t} b(\vr)  \Div \vu \varphi\, \dxdt - \int_{\Omega_0}  \vr_0 B(\vr_0)  \varphi (0)\, \dx
\eF
 for any $\varphi \in C^1_c([0,T) \times \mathbb{R}^3)$, and any  $b \in L^\infty \cap C [0, \infty)$ such that { $b(0)=0$} and 
 $B(\vr) = B(1) + \int_1^\vr \frac{b(z)}{z^2} {\rm d}z $
where we have $\vr \geq 0$ a.e. in $(0,T) \times \mathbb{R}^3$.

The momentum equation \eqref{i1b} is transferred to the following integral identity
\begin{align} \nonumber
&\int_0^T \int_{\Omega_t} \left( \vr \vu \cdot \partial_t \vph + \vr [\vu \otimes \vu] : \Grad \vph + \frac{1}{\ep^2}p(\vr,\vt) \Div \vph
- \tn{S} (\vt, \Grad \vu) : \Grad \vph + \frac{1}{\ep}\vr \nabla_x F\cdot \vph \right) \dxdt \\
& \qquad = - \int_{\Omega_0} \vr_0 \vu_0 \cdot \vph (0, \cdot) \ \dx,\label{m3}
\end{align}
for any test function $\vph \in C^1_c (\overline{Q_T} ; \mathbb{R}^3)$ 
such that $\vph(T,\cdot) =0$ and 
$
\vph \cdot \vc{n}|_{\Gamma_\tau} = 0$ for any $\tau \in [0,T].$
Moreover,
\bFormula{m50}
\vu,\nabla_x\vu \in L^2(Q_T; \mathbb{R}^3) \ \mbox{and}\ (\vu - \vc{V}) \cdot \vc{n}  (\tau , \cdot)|_{\Gamma_\tau}  = 0 \ \mbox{for a.a.}\ \tau \in [0,T],
\eF

The entropy balance \eqref{i1c} is rewritten in the form of equation
	\begin{align} \nonumber
	& \int_0^T \int_{\Omega_t} \vr s (\partial_t \varphi + \vu \cdot \Grad \varphi ) \dxdt 
	- \int_0^T \int_{\Omega_t} \frac{\kappa(\vt) \Grad \vt \cdot \Grad \varphi }{\vt} \dxdt \\
	& 
     \quad + \langle \sigma_\ep;\varphi \rangle =
    - \int_{\Omega_0} \vr_0 s_0 \varphi (0) \dx\label{m5}
	\end{align}
for all $\varphi \in C^1 (\overline{Q_T})$ such that $\varphi(T,\cdot) =0$
.

Finally, the energy inequality has to cover the movement of the domain, hence we get
\begin{align} \nonumber 
&\int_{\Omega_\tau}\left(\frac{\ep^2}{2} \vr |\vu|^2 + \vr e - \ep\vr F\right)(\tau,\cdot) \dx\leq \int_{\Omega_0} \left(\frac{\ep^2}{2} (\vr_0 \vu_0^2) + \vr_0 e_0 - \ep\vr_0 F - \ep^2 \vr_0 \vu_0\cdot \vc{V}(0)\right)\dx\\
&-\ep^2 \int_0^\tau\int_{\Omega_t}\left( \vr(\vu\otimes\vu):\Grad \vc{V} 
-\mathbb S:\Grad \vc{V} + \vr \vu\cdot \partial_t \vc{V}
 + \frac{1}{\varepsilon} \varrho \nabla_x F \cdot \vc{V}
\right) \dxdt  + \ep^2 \int_{\Omega_\tau}\vr \vu \cdot \vc{V}(\tau,\cdot)\dx\label{m6}
\end{align}
for a.a. $\tau\in(0,T)$. Notice that the term containing $\Div \vV$ vanishes due to assumption \eqref{divV} and \eqref{m6} is an inequality, what differs this formulation from the one in \cite{FeNo6} (see \cite{KMNW2}).

Let us remark that when writing $f \in L^\infty(0,T;L^q(\Omega_t))$ for some $q \in[1,\infty)$ we mean that the mapping $t\to \|f(t,\cdot)\|_{L^q(\Omega_t)}$ is measurable and bounded function on time interval $[0,T]$.

We have
\begin{Theorem}\label{t:existence}
Let $\Omega_0 \subset \mathbb{R}^3$ be a bounded domain of class $C^{2 + \nu}$ with some $\nu >0$, and let $\vc{V} \in C^1([0,T]; C^{3}_c (\mathbb{R}^3;\mathbb{R}^3))$ be given. Assume that hypothesis \eqref{i4}--\eqref{pm3} are satisfied, let $F \in W^{1,\infty}(\R^3)$ and let $\ep > 0$ but sufficiently small s.t. $\varrho_{0,\ep}\geq 0$, $\vartheta_{0,\ep}>0$.

Then the problem \eqref{i1a}--\eqref{i1d} with boundary conditions \eqref{i6}, \eqref{b1}, \eqref{b2} and initial conditions \eqref{initial_1}--\eqref{initial_3}, where the entropy production rate $\sigma_\ep$ satisfies \eqref{sigma1}, admits a variational solution on any finite time interval $(0,T)$. Namely,  the trio $(\vr_\ep,\vu_\ep,\vt_\ep)$ satisfies \eqref{m2}--\eqref{m6}.
Moreover 
\begin{itemize}
\item $\vr_\ep \in L^\infty(0,T; L^{\frac 53}(\Omega_t))$, $\vr_\ep \geq 0$, 
	$\vr_\ep \in L^q(Q_T)$ for certain $q>\frac{5}{3}$, 
\item $\vu_\ep,\, \nabla_x \vu_\ep \in L^2(Q_T)$, $\vr_\ep \vu_\ep \in L^\infty(0,T;L^1(\Omega_t))$,
\item  $\vt_\ep>0$ a.a. on $Q_T$, 
$\vt_\ep \in L^\infty(0,T;L^4 (\Omega_t))$, $\vt_\ep,\,\nabla_x \vt_\ep \in L^2(Q_T),$ and  $\log\vt_\ep,\,\nabla_x \log\vt_\ep \in L^2(Q_T),$ 
\item $\vr_\ep s(\vr_\ep,\vt_\ep),\ \vr_\ep s(\vr_\ep,\vt_\ep) \vu_\ep, \ \frac{\vq(\vt_\ep)}{\vt_\ep} \in L^1(Q_T)$,
\end{itemize}
\end{Theorem}
\begin{proof}
This theorem for $\ep = 1$ and $F = 0$ has been recently proved in \cite{KMNW2}. In order to accommodate nonzero forcing $F \neq 0$ we do not need any additional technique in the proof, we just handle lower order terms $\vr \nabla_x F$ in the momentum equation and $\vr F$ in the energy inequality. In particular one easily observes that using the penalization technique as in \cite{KMNW2}, both these terms become equal to zero on the "solid" part of the artificial domain $B$ when we choose the initial density $\vr_0$ to be equal to zero on the solid part. The scaling by $\ep$ does not represent any additional difficulties in the proof of existence of variational solutions.
\end{proof}

\section{Target system and main result}\label{s:target}
First, we introduce 
\begin{equation}\label{eq:Gdef}
G(t) := \fint_{\Omega_t} \vV(t,x) \cdot \nabla_x F(x) \dx = \frac{1}{|\Omega_t|} \int_{\Omega_t} \vV(t,x) \cdot \nabla_x F(x) \dx.
\end{equation}
We claim that the following version of the Oberback-Boussinesq system is recovered as the low Mach number limit
\bFormula{OB_system1}
\Div \vU = 0
\eF
\bFormula{OB_system2}
\overline \vr (\partial_t \vU + \Div (\vU\otimes \vU)) + \nabla_x \Pi - \mu(\overline \vartheta) \Delta_x \vU = r \nabla_x F
\eF
\bFormula{OB_system3}
\overline \varrho\, \overline{c_p} \left(\partial_t \Theta + \Div (\Theta \vU)\right) - \kappa(\vtb)\Delta_x\Theta - \overline \alpha\, \vrb\,\vtb \vU\cdot\nabla_x F   = -\overline \alpha\,\vrb\,\vtb G
\eF
\bFormula{OB_system4}
r + \overline \varrho\, \overline \alpha \Theta  = 0.
\eF
This system is considered on a time dependent domain $Q_T$ and supplemented with boundary conditions 
 \begin{equation}\label{bc_OB_1}
 (\vU(\tau, \cdot) - \vV(\tau,\cdot))\cdot\vn|_{\Gamma_\tau} = 0, \quad
 ((\nabla_x \vU+\nabla_x^t \vU)\vc{n}) \times \vc{n} |_{\Gamma_t} = 0, \quad \nabla_x\Theta\cdot\vn|_{\Gamma_t}=0
 \end{equation}
and initial data 
	\begin{equation}\label{in_OB_0}
    \vU(0,\cdot) = \vU_0, \quad \Theta_0 = \Theta(0,\cdot) \quad \mbox{ in } \Omega_0
    \end{equation}
Here $\overline \alpha = \frac{1}{\vrb} \frac{\partial_\vartheta p}{\partial_\varrho p}(\vrb,\vtb)$ and $\overline{c_p} = \de_\vartheta e(\vrb, \vtb) + \overline{\alpha} \frac{\vtb}{\vrb} \de_\vartheta p(\vrb,\vtb)$.
   
Note that the difference with respect to classical Oberback-Boussinesq system is the presence of additional forcing term in the equation for temperature variation $\Theta$. As we will see later, the presence of this term in the system is related to the fact that unlike in the case of a fixed domain, it is no longer possible to assume
\begin{equation*}
\int_{\Omega_t} F(x) \dx = 0
\end{equation*}
for all $t \in [0,T)$. It is however interesting to notice, that $G \equiv 0$ whenever the motion of the physical domain is perpendicular to the gradient of the potential $F$ and in that case one ends up with a usual Oberback-Boussinesq system.

We also note that $G(t)$ can be written as 
\begin{equation*}
G(t) = \fint_{\Omega_t} \Div\left(\vV(t,x) F(x)\right) \dx = \frac{1}{|\Omega_t|} \int_{\Gamma_t} F\vV\cdot\vn \dx
\end{equation*}
and combined with the boundary condition \eqref{bc_OB_1} we have
\begin{equation*}
G(t) = \fint_{\Omega_t} \Div\left(\vU(t,x) F(x)\right) \dx = \fint_{\Omega_t} \vU(t,x)\cdot\nabla_x F(x) \dx, 
\end{equation*}
yielding that \eqref{OB_system3} can be written as
\begin{equation*}
\overline \varrho\, \overline{c_p} \left(\partial_t \Theta + \Div (\Theta \vU)\right) - \kappa(\vtb)\Delta_x\Theta  = \overline \alpha\, \vrb\,\vtb\left( \Div(F\vU) - \fint \Div(F\vU)\right).
\end{equation*}

\bigskip
\begin{Definition}\label{def:OB_sol}
We say that $\vU$, $\Theta$ is a weak solution to the Oberback-Boussinesq system 
\eqref{OB_system1}--\eqref{OB_system4}
if the following holds:
\begin{itemize}
\item $\Div \vU (\tau, \cdot )=0 $ and $(\vU(\tau, \cdot) - \vV(\tau,\cdot))\cdot \vn|_{\Gamma_\tau} =0$ for a.a. $\tau \in (0,T)$,
\item The equation
\begin{equation*}
\begin{split}
\int_0^T \int_{\Omega_t} & \overline \varrho \vU \cdot \partial_t \vph + \overline{\varrho} (\vU \otimes \vU):\Grad \vph - \mu(\overline \vartheta)\left(\nabla_x\vU + \nabla_x^t \vU\right)  : \Grad \vph 
\\ & = - \int_0^T \int_{\Omega_t} r \nabla_x F \cdot \vph \dxdt 
- \int_{\Omega_0} \vU_0 \cdot \vph(0,\cdot) \dx
\end{split}
\end{equation*}
holds for all $\vph\in C^{\infty}_c(\overline{Q_T})$ such that $\vph(T,\cdot) = 0$,  $\Div \vph = 0$ and $\vph \cdot \vc{n}|_{\Gamma_t} = 0$,
\item 
equation \eqref{OB_system3} is satisfied a.a. in $Q_T$ and $\nabla_x \Theta \cdot \vn|_{\Gamma_t} = 0$ in a sense of traces for a.a. $\tau \in (0,T)$,
\item Boussinesq relation \eqref{OB_system4} holds,
\item $\vU,\ \nabla_x \vU \in L^2(Q_T)$, ${\rm ess\,}\sup_{t\in (0,T)} \|\vU(t,\cdot)\|_{L^2(\Omega_t)} < c$,
\item the mapping $t \mapsto \| \Theta(t,\cdot)\|_{L^q(\Omega_t)}$ belongs to $W^{1,q}_{loc}((0,T]) \cap C([0,T])$  and  the mapping $t \mapsto \| \Theta(t,\cdot)\|_{W^{2,q}(\Omega_t)}$ belongs to $L^{q}_{loc}((0,T])$ for certain $q>1$.
\end{itemize}
\end{Definition}

\begin{Theorem}\label{t:main}
Let $\Omega_0 \subset \mathbb{R}^3$ be a bounded domain of class $C^{2 + \nu}$ with some $\nu >0$, and let $\vc{V} \in C^1([0,T]; C^{3}_c (\mathbb{R}^3;\mathbb{R}^3))$ be given and satisfy \eqref{divV}. Assume that hypothesis \eqref{i4}--\eqref{pm3} are satisfied, let $F \in W^{1,\infty}(\R^3)$ and let $\ep > 0$ but sufficiently small s.t. $\varrho_{0,\ep}\geq 0$, $\vartheta_{0,\ep}>0$.

Let the trio $(\vr_\ep,\vu_\ep,\vt_\ep)$ be a variational solution to  the problem \eqref{i1a}--\eqref{i1d} with boundary conditions \eqref{i6}, \eqref{b1}, \eqref{b2} and initial conditions \eqref{initial_1}--\eqref{initial_3}  on any finite time interval $(0,T)$, and where entropy production rate satisfies \eqref{sigma1}.

Then 
	\begin{equation}
    {\rm ess}\sup_{t\in (0,T)} \| \varrho_\ep - \overline\varrho\|_{L^q(\Omega_t)} \to 0 \quad \mbox{ as } \ep \to 0 \mbox{ for certain }q>1,
    \end{equation}
and, for a suitable subsequence, 	
    \begin{equation}
	\vu_\ep \weak \vU \quad \mbox{ weakly in } L^2(Q_T) 
    \mbox{ for } \ep \to 0,
    \end{equation}
     \begin{equation}
	\nabla_x\vu_\ep \weak \nabla_x\vU \quad \mbox{ weakly in } L^2(Q_T) 
    \mbox{ for } \ep \to 0,
    \end{equation}
      \begin{equation}
	\frac{\vartheta_\ep - \overline{\vartheta}}{\ep} \weak \Theta \quad \mbox{ weakly in } L^q(Q_T) 
    \mbox{ for } \ep \to 0 \mbox{ with certain }q>1,
    \end{equation}
 where the couple $\vU$ and $\Theta$ is a weak solution  according to Definition~\ref{def:OB_sol} to the Oberbeck-Boussinesq system \eqref{OB_system1}--\eqref{OB_system4} with boundary condition \eqref{bc_OB_1} and initial data \eqref{in_OB_0} 
 \begin{equation}\label{in_OB_1}
 \vU_0 = \vc{H}_0 [\vu_0], \quad
 \Theta_0 = \frac{\overline{\vartheta}}{\overline c_p}\left(\partial_\varrho s(\vrb,\vtb) \varrho^{(1)}_0 + \partial_\vartheta s(\vrb,\vtb) \vartheta^{(1)}_0 + \overline \alpha F\right) \quad
 \mbox{ in } \Omega_0,
     \end{equation}
 where $\vu_{0,\ep} \weakstar \vu_0$,  $\vr^{(1)}_{0,\ep} \weakstar \vr^{(1)}_0$, $\vt^{(1)}_{0,\ep} \weakstar \vt^{(1)}_0$	weakly$^*$ in $L^\infty(\Omega_0)$ for $\ep \to 0$.
\end{Theorem}
In the above  $\vc{H}_0 $ stand for Helmholtz projection onto the space of solenoidal functions on $\Omega_0$.

\section{Low Mach number limit}\label{s:lowmach}



\subsection{Uniform estimates}
\label{sec:uni:est}
Before stating uniform estimates let us recall here some basic notations and results which we need in proving our convergence results. We refer to  \cite{FeNo6}.

First of all we fix a smooth function $\chi \in C^\infty_c ( (0,\infty)\times(0,\infty) )$ such that 
	$0\leq \chi \leq 1, \ \chi=1$  on the set $ {\cal{O}}_{ess}$, where we define
\begin{equation*}\label{ess-set}
	 {\cal{O}}_{ess} =  [ \overline\varrho/2 , 2\overline\varrho ] \times [\overline{\vartheta}/2, 2 \overline{\vartheta}],\quad  
	 {\cal{O}}_{res} = (0,\infty)^2\setminus {\cal{O}}_{ess} .
 \end{equation*}
Namely, ${\cal{O}}_{ess}$ is a neighborhood of the target density and temperature.

Then, we introduce the decomposition on essential and residual part of a measurable function $h$ as follows: we define the decomposition
	\begin{equation*}\label{ess-def}
	h = [h]_{ess} + [h]_{res},\qquad\mbox{ with }\quad  [ h]_{ess} := \chi(\varrho_\ep,\vartheta_\ep) h\,,\quad \  [h]_{res} = (1-\chi(\varrho_\ep,\vartheta_\ep))h\,.
	\end{equation*}

Since $\Div \vc{V} = 0$, from the energy inequality \eqref{m6} and the entropy balance \eqref{m5} we get
\begin{multline}\label{uni_est1}
\left[\int_{\Omega_t} \frac12\varrho |\vu|^2  + \frac 1{\varepsilon^2} \left(H_{\overline \vartheta}(\varrho, \vartheta) - (\varrho - \overline \varrho)\frac{\partial H_{\overline \vartheta}(\overline \varrho, \overline \vartheta)}{\partial \varrho} - H_{\overline \vartheta}(\overline \varrho, \overline \vartheta)\right) - \frac{\varrho - \overline \varrho}{\varepsilon} F\dx\right]_{t=0}^{t=\tau}
\\
+\frac{\overline\vartheta}{\ep^2} \sigma_\ep \left[\cup_{s\in[0,t]} \Omega_s \right]
\leq 
-\int_{0}^\tau \int_{\Omega_t} \varrho \vu \otimes \vu: \nabla_x \vV - \tn{S}:\nabla_x \vV - \varrho \vu \cdot \partial_t \vV\dxdt 
\\
+ \left[\int_{\Omega_t}\varrho \vu \cdot  \vV(t,\cdot)\right]_{t=0}^{t=\tau} - \int_0^\tau \int_{\Omega_t} \frac{\varrho - \overline \varrho}{\varepsilon}\nabla_x F \cdot \vV\dxdt
\end{multline}
where 
$$H_{\overline{\vartheta}}(\vr_\ep,\vt_\ep) = \varrho_\ep(e(\varrho_\ep,\vt_\ep) - \overline{\vartheta}s(\varrho_\ep,\vt_\ep))$$ 
(see \cite[Chapter 2.2.3]{FeNo6})
is a Helmholtz function.

By \eqref{uni_est1} and with \cite[Lemma 5.1]{FeNo6} we obtain the following set of estimates. The details can be found in \cite[Chapter~5]{FeNo6}. 
\begin{Lemma}\label{l:estimates} 
Let assumptions of the Theorem~\ref{t:existence} be satisfied. 
Let $\{(\varrho_\ep,\vu_\ep,\vartheta_\ep)\}_{\ep>0}$ be a sequence of weak solutions obtained in Theorem~\ref{t:existence}. Then the following estimates hold 
\label{estimates}
\begin{equation}\label{est_1}
{\rm ess\,}\sup_{t\in (0,T)} \int_{\Omega_t} [ 1(t) ]_{res} \dx \leq \ep^2 c,
\end{equation}
\begin{equation}\label{est_2}
{\rm ess\,}\sup_{t\in (0,T)} \left\| \left[ \frac{\varrho_\ep - \overline\varrho}{\ep}  \right]_{ess} (t) \right\|_{L^2(\Omega_t)} \leq c,
\end{equation}
\begin{equation}\label{est_3}
{\rm ess\,}\sup_{t\in (0,T)} \left\| \left[ \frac{\vartheta_\ep - \overline\vartheta}{\ep}  \right]_{ess} (t) \right\|_{L^2(\Omega_t)} \leq c,
\end{equation}
\begin{equation}\label{est_4}
{\rm ess\,}\sup_{t\in (0,T)}\int_{\Omega_t} 
\left( [\varrho_\ep]^{\frac{5}{3}}_{res} + [\vartheta_\ep]^4_{res}   \right) (t) \dx \leq \ep^2 c,
\end{equation}
\begin{equation}\label{est_5}
{\rm ess\,}\sup_{t\in (0,T)} \left\|\sqrt{\varrho_\ep} \vu_\ep   \right\|_{L^2(\Omega_t)} \leq c,
\end{equation}
\begin{equation}\label{est_6}
\sigma_\ep [Q_T] \leq \ep^2 c,
\end{equation}
\begin{equation}\label{est_7}
\int_0^T  \left\|  \vu_\ep (t)\right\|^2_{W^{1,2}(\Omega_t)} \dt  \leq  c,
\end{equation}
\begin{equation}\label{est_8}
\int_0^T  \left\| \left( \frac{\vartheta_\ep - \overline\vartheta}{\ep}\right) (t) \right\|^2_{W^{1,2}(\Omega_t)} \dt  \leq  c,
\end{equation}
\begin{equation}\label{est_9}
\int_0^T  \left\| \left( \frac{\log(\vartheta_\ep) - \log(\overline\vartheta)}{\ep}\right) (t) \right\|^2_{W^{1,2}(\Omega_t)} \dt  \leq  c,
\end{equation}
\begin{equation}\label{est_10}
\int_0^T  \left\| \left[ \frac{ \varrho_\ep s(\varrho_\ep, \vartheta_\ep)}{\ep}\right]_{res} (t) \right\|^q_{L^{q}(\Omega_t)} \dt  \leq  c,
\end{equation}
\begin{equation}\label{est_11}
\int_0^T  \left\| \left[ \frac{ \varrho_\ep s(\varrho_\ep, \vartheta_\ep)}{\ep}\right]_{res}  \vu_\ep (t) \right\|^q_{L^{q}(\Omega_t)} \dt  \leq  c,
\end{equation}
\begin{equation}\label{est_12}
\int_0^T  \left\| \left[ \frac{ \kappa(\vartheta_\ep)}{\vartheta_\ep}\right]_{res}  \left( \frac{ \nabla_x \vartheta_\ep}{\ep}\right)(t) \right\|^q_{L^{q}(\Omega_t)} \dt  \leq  c.
\end{equation}
\end{Lemma}

\bigskip
Let us introduce the following notation 
$$
\varrho_\varepsilon^{(1)} = \frac{\varrho_\varepsilon - \overline \varrho}{\varepsilon},\qquad\vartheta_\varepsilon^{(1)} = \frac{\vartheta_\varepsilon - \overline \vartheta}{\varepsilon}.
$$

We assume that $\varrho_\ep$ is extended to the whole space $\R^3$ by a constant $\overline \varrho$. Similarly, we extend also the velocity and the temperature to the whole space by a standard extension $E_t:W^{1,2}(\Omega_t)\mapsto W^{1,2}(\R^3)$ which is uniformly bounded with respect to $t\in [0,T]$ (we refer to \cite{AdFo}) as the fluid domain is regular at each time. We skip $E_t$ in the notation so from now on it holds that $\vartheta_\varepsilon = E_t \vartheta_\varepsilon$ and $\vu_\varepsilon = E_t \vu_\varepsilon$.

As a consequence of the above Lemma~\ref{l:estimates} we obtain the following convergences (for details see \cite[Chapter~5.3]{FeNo6})
\bFormula{conv_001}
\varrho_\varepsilon  - \overline \varrho \to 0  \mbox{ in }L^\infty(0,T;L^{5/3}(\R^3)),
\eF
\bFormula{conv_002}
\left( t \to \int_{\Omega_t}(\varrho_\varepsilon - \overline \varrho) \varphi\dx \right)   \to 0  \mbox{ in }C([0,T]) \mbox{ for all }   \varphi \in L^{r'}(\Omega_t)  \mbox{ with }r\in [1,5/3),
\eF
\bFormula{conv_003}
\vartheta_\varepsilon - \overline \vartheta \to 0 \mbox{ in }L^\infty(0,T;L^2(\R^3)),
\eF
\bFormula{conv_004}
\varrho_\varepsilon^{(1)} \weakstar \varrho^{(1)} \mbox{ weakly}^*\mbox{ in } L^\infty(0,T;L^{5/3}(\R^3)),
\eF
\bFormula{conv_005}
\vartheta_\varepsilon^{(1)} \weak \vartheta^{(1)}  \mbox{ weakly in } L^2(Q_T),
\eF
\bFormula{conv_006}
\nabla_x \vartheta_\varepsilon^{(1)} \weak \nabla_x \vartheta^{(1)}  \mbox{ weakly in } L^2(Q_T),
\eF
\bFormula{conv_007}
\vu_\varepsilon \weak \vU  \mbox{ weakly in } L^2(0,T;L^6(\R^3)).
\eF
\bFormula{conv_008}
\vu_\varepsilon \weak \vU  \mbox{ weakly in } L^2(0,T;W^{1,2}(\R^3)),
\eF
\bFormula{conv_009}
\left[ \frac{\vt_\ep - \vtb}{\ep} \right]_{ess} 
\weakstar \vartheta^{(1)}
\mbox{ weakly}^* \mbox{ in } L^\infty(0,T;L^2(\R^3)),
\eF
\bFormula{conv_011}
\left[ \frac{\vr_\ep s(\vr_\ep,\vt_\ep)}{\ep} \right]_{res}
\to 0 \mbox{ in } L^q(Q_T) \mbox{ for certain }q>1,
\eF
\bFormula{conv_012}
\left[ \frac{\vr_\ep}{\ep}  \right]_{res}
\to 0 \mbox{ in } L^\infty(0,T; L^{\frac{5}{3}}(\R^3)),
\eF
\bFormula{conv_013}
[\vr_\ep]_{ess} \frac{[s(\vr_\ep,\vt_\ep)]_{ess} - s(\vrb, \vtb)}{\ep}
\weakstar
\vrb \left( \partial_\vr s (\vrb,\vtb) \vr^{(1)} 
+ \partial_\vr s (\vrb,\vtb) \vt^{(1)} \right)
\mbox{ weakly}^* \mbox{ in }L^\infty(0,T; L^2 (\R^3)),
\eF
\bFormula{conv_0130}
[\vr_\ep]_{ess} \frac{[s(\vr_\ep,\vt_\ep)]_{ess} - s(\vrb, \vtb)}{\ep} \vu_\ep
\weak
\vrb \left( \partial_\vr s (\vrb,\vtb) \vr^{(1)} 
+ \partial_\vt s (\vrb,\vtb) \vt^{(1)} \right) \vU
\mbox{ weakly in }L^2(0,T; L^{\frac{3}{2}} (\R^3)),
\eF
\bFormula{conv_014}
\left[\frac{\kappa(\vt_\ep)}{\vt} \right]_{ess} 
\nabla_x \left( \frac{\vt_\ep - \vtb}{\ep}\right)
\weak
\frac{\kappa(\vtb)}{\vtb} \nabla_x \vt^{(1)} \mbox{ weakly in }
L^2(0,T;L^2(\R^3)),
\eF
\bFormula{conv_015}
\left[ \frac{\kappa(\vt_\ep)}{\vt_\ep} \right]_{res}
\nabla_x \left( \frac{\vt_\ep}{\ep} \right) 
\to 0 \mbox{ in } L^q (Q_T) \mbox{ for certain }q>1,
\eF
\bFormula{conv_017}
\vr_\ep \vu_\ep \otimes \vu_\ep \weak \overline{\vr \vU \otimes \vU }
\mbox{ weakly in } L^2( 0,T; L^{\frac{30}{29}}(\R^3)),
\eF
\bFormula{conv_018}
\tn{S}_\ep \weak \mu(\vtb)(\nabla_x \vU + \nabla_x^t \vU) \mbox{ weakly in } L^q(Q_T) \mbox{ for certain } q>1.
\eF

\subsection{Limit in the continuum equation}
We recall the Reynolds transport theorem:
\begin{Theorem}\label{t:transport}
Let a general function $f = f(t,x)$ belong to $C^1((t_1,t_2);W^{1,\infty}(\Omega_t))$
and let $\vV\in C^{1}(\mathbb R^+\times \mathbb R^3)$. Then for each $t\in (t_1,t_2)$ there exists a finite derivative
$$
\frac{{\rm d}}{{\rm d}t}\int_{\Omega_t}f(t,x)\dx = \int_{\Omega_t}\left( \partial_t f(t,x) + \Div (f\vV)(t,x)\right) \dx.
$$
\end{Theorem}

We proceed with the limit in the continuum equation. By \eqref{conv_001} and  \eqref{conv_008} we obtain the boundary condition
\begin{equation}\label{eq:UBC}
\vU\cdot \vn = \vV\cdot\vn \qquad \text{ on } \Gamma_\tau
\end{equation}
in the sense of traces. Moreover, passing with $\ep \rightarrow 0$ in \eqref{m2} with the choice $b(\vr) \equiv 0$, $B(\vr) \equiv 1$ we achieve
\begin{equation*}
\vrb \intint{(\partial_t \varphi + \vU \cdot \nabla_x \varphi)}  = - \vrb\int_{\Omega_0} \varphi(0,\cdot) \dx
\end{equation*}
for all $\varphi \in C^1_c([0,T)\times \mathbb{R}^3)$.  We use the transport Theorem \ref{t:transport} to conclude that 
$$\Div \vU = 0 \quad \mbox{ a.e. in } Q_T.$$ Indeed,
\begin{align}
&\intint{(\partial_t \varphi + \vU \cdot \nabla_x \varphi)} = \intint{(\partial_t \varphi - \Div \vU \varphi)} + \int_0^T \int_{\Gamma_t} \vU\cdot\vn \varphi\ \dS \dt \nonumber \\
&=\intint{ (\partial_t \varphi - \Div \vU \varphi)} + \int_0^T \int_{\Gamma_t} \vV\cdot\vn \varphi \dS\dt = \intint{(\partial_t \varphi - \Div \vU \varphi + \Div(\vV\varphi))} \nonumber \\
&=- \intint{\Div \vU \varphi} + \int_0^T\left(\frac{\mathrm{d}}{\mathrm{dt}} \int_{\Omega_t} \varphi \dx \right) \dt=  -\intint{ \Div \vU \varphi} - \int_{\Omega_0} \varphi(0,\cdot) \dx \nonumber
\end{align}
and thus $\intint{ \Div \vU \varphi} = 0$ for all $\varphi \in C^1_c([0,T)\times \mathbb{R}^3)$.

\subsection{Limit in the balance of entropy}

Similarly as in \cite[Section 5.3.2]{FeNo6}, we deduce that by convergences from Section~\ref{sec:uni:est} the balance of entropy \eqref{m5} in the limit $\ep \to 0$ may take the following form
\begin{align}\label{limitentropy}
&\int_0^T\int_{\Omega_t} \overline \varrho \left(\partial_\varrho s(\vrb,\vtb) \varrho^{(1)} + \partial_\vartheta s(\vrb,\vtb) \vartheta^{(1)}\right)(\partial_t\varphi + \vU \cdot \nabla_x \varphi)\dxdt \\ \nonumber
&\quad - \int_0^T\int_{\Omega_t} \frac{\kappa(\overline \vartheta)}{\overline\vartheta} \nabla_x \vartheta^{(1)}\cdot \nabla_x \varphi\dxdt = -\int_{\Omega_0} \overline \varrho \left(\partial_\varrho s(\vrb,\vtb) \varrho^{(1)} + \partial_\vartheta s(\vrb,\vtb) \vartheta^{(1)}\right)_0 \varphi(0,\cdot)\dx.
\end{align}
for any $\varphi \in C^1 (\overline Q_T)$, $\varphi(T,\cdot) = 0$.

Further, we multiply the momentum equation \eqref{m3} by $\varepsilon $ and we let $\varepsilon \to 0$ in order to obtain
\begin{equation}\label{limmomeps}
\int_0^T\int_{\Omega_t}\left(\partial_\varrho p(\overline \varrho, \overline \vartheta) \varrho^{(1)} + \partial_{\vartheta}p(\overline \varrho, \overline \vartheta)\vartheta^{(1)}\right)\Div \vph\dxdt 
= -\int_0^T \int_{\Omega_t}\overline \varrho \nabla_x F\cdot \vph \dxdt .
\end{equation}
for any $\vph \in C^1 (\overline Q_T)$, $\vph(T,\cdot) = 0$, $\vph \cdot \vc{n}|_{\Gamma_\tau} = 0$. 

We define 
\begin{equation*}
C(t):=  \fint_{\Omega_t}F(x)\, {\rm d}x
\end{equation*}
and we can assume, without loss of generality, that $C(0) =0$. The conservation of mass together with the assumption \eqref{initial_2} yields $\int_{\Omega_t} \varrho^{(1)} = 0$, whereas the same property for the temperature $\int_{\Omega_t}\vartheta^{(1)} = 0$ is a consequence of \eqref{limitentropy} and \eqref{initial_2}. These properties and \eqref{limmomeps} yield
\begin{equation}
\label{rhojedna}
\varrho^{(1)} = -\frac{\partial_\vartheta p(\vrb,\vtb)}{\partial_{\varrho}p(\vrb,\vtb)} \vartheta^{(1)} + \frac{\overline \varrho F}{\partial_{\varrho}p(\vrb,\vtb)} - \frac{\overline \varrho C(t)}{\partial_\varrho p(\vrb,\vtb)}.
\end{equation}
In order to simplify notation we introduce
\begin{equation*}\label{cp}
c_p(\varrho,\vartheta) = \de_\vartheta e(\varrho, \vartheta) + \alpha (\varrho,\vartheta) \frac{\vartheta}{\varrho} \de_\vartheta p(\vr,\vt),  \qquad \overline{c_p} = c_p(\vrb,\vtb),
\end{equation*}
\begin{equation*}\label{alpha}
\alpha(\varrho,\vartheta) = \frac{1}{\varrho} \frac{\partial_\vartheta p(\varrho,\vartheta)}{\partial_\varrho p(\varrho,\vartheta)}, \qquad \overline{\alpha} = \alpha(\vrb,\vtb).
\end{equation*}
We plug \eqref{rhojedna} into \eqref{limitentropy}, we multiply it by $\overline \vartheta$ and we employ Maxwell and Gibbs relations in order to get
\begin{align} \label{limittheta}
&\int_0^T \int_{\Omega_t} \overline \varrho \left(\overline{c_p} \vartheta^{(1)} - \overline \vartheta \overline{\alpha} (F - C)\right)\left(\partial_t \varphi + \vU\cdot \nabla_x \varphi\right)\dxdt\\ \nonumber
& \quad - \int_0^T\int_{\Omega_t} \kappa(\overline \vartheta) \nabla _x\vartheta^{(1)}\cdot \nabla_x \varphi\dxdt = -\int_{\Omega_0} \overline\vartheta\overline \varrho \left(\partial_\varrho s(\vrb,\vtb) \varrho^{(1)}_0 + \partial_\vartheta s(\vrb,\vtb) \vartheta^{(1)}_0\right) \varphi(0,\cdot)\dx 
\end{align}
for any $\varphi \in C^1 (\overline Q_T)$, $\varphi(T,\cdot) = 0$.

We use the transport Theorem \ref{t:transport} and \eqref{eq:UBC} to observe that for all $\varphi \in C^1 (\overline Q_T)$, $\varphi(T,\cdot) = 0$ it holds
\begin{align} \nonumber
&\int_0^T \int_{\Omega_t} F(x)(\de_t \varphi + \vU\cdot\nabla_x\varphi)\dxdt \\ \nonumber 
&\quad= \int_0^T\int_{\Omega_t} \left( F\de_t\varphi - \varphi \Div(F\vU) \right)\dxdt + \int_0^T \int_{\Gamma_t} F\varphi\vU\cdot\vn\,{\rm d} S \dt \\
&\quad = \int_0^T \int_{\Omega_t} \left(\de_t(F\varphi) + \Div(F\varphi\vV) - \varphi \vU\cdot\nabla_x F\right) \dxdt  \nonumber \\
&\quad = \int_0^T \left(\frac{{\rm d}}{\dt} \int_{\Omega_t} F\varphi \dx\right) \dt - \int_0^T\int_{\Omega_t}\varphi \vU\cdot\nabla_x F \dxdt \nonumber \\ 
&\quad = -\int_{\Omega_0} F(x)\varphi(0,x) \dx - \int_0^T\int_{\Omega_t}\varphi \vU\cdot\nabla_x F \dxdt,\nonumber
\end{align}
whereas for the same class of test functions $\varphi$ we have (recalling \eqref{eq:Gdef} where we defined $G(t)$)
\begin{align} \nonumber
&\int_0^T \int_{\Omega_t} C(t)(\de_t \varphi + \vU\cdot\nabla_x\varphi)\dxdt \\ \nonumber 
&\quad = \int_0^T\int_{\Omega_t}\left( C\de_t\varphi - \varphi \Div(C\vU)\right)\dxdt + \int_0^T \int_{\Gamma_t} C\varphi\vU\cdot\vn\,{\rm d} S \dt \\
&\quad = \int_0^T \int_{\Omega_t}\left( \de_t(C\varphi) + \Div(C\varphi\vV) - \de_t C \varphi \right) \dxdt  \nonumber \\
&\quad = \int_0^T \frac{{\rm d}}{\dt} \int_{\Omega_t} C(t)\varphi(t,x) \dxdt - \int_0^T\int_{\Omega_t}G(t)\varphi(t,x) \dxdt = - \int_0^T\int_{\Omega_t}G(t)\varphi(t,x) \dxdt.\nonumber
\end{align}
Here we have used $C(0) = 0$ and the observation that 
\begin{equation*}
\frac{{\rm d}}{\dt} C(t) = \frac{{\rm d}}{\dt}\fint_{\Omega_t} F(x) \dx = \fint_{\Omega_t} \Div (F\vV)\dx = \fint_{\Omega_t} \vV\cdot \nabla_x F \dx = G(t).
\end{equation*}


Taking $\Theta = \vartheta^{(1)}$, we deduce from \eqref{limittheta} the following equation
$$
\overline \varrho\, \overline c_p \partial_t \Theta + \overline \varrho\, \overline c_p\Div (\Theta \vU) - \kappa(\vtb) \Delta_x \Theta - \overline \alpha\,\overline \varrho\, \overline \vartheta \vU  \cdot\nabla_x F = - \overline \alpha\,\overline \varrho\, \overline \vartheta G,
$$
with the initial data
\begin{equation}\label{ini:theta}
\overline c_p\Theta_0 = \overline{\vartheta}\left(\partial_\varrho s(\vrb,\vtb) \varrho^{(1)}_0 + \partial_\vartheta s(\vrb,\vtb) \vartheta^{(1)}_0 + \overline \alpha F\right).
\end{equation}

Finally, we define 
\begin{equation*}
r := \varrho^{(1)}  - \frac{\overline \varrho}{\partial_\varrho p(\vrb,\vtb)}(F-C).
\end{equation*}
Then, \eqref{rhojedna} yields the Boussinesq relation
$$
r + \overline \varrho \,\overline \alpha \Theta = 0.
$$

\subsection{Limit in the momentum equation}

Since $\Div \vU = 0$ we may take as a test function $\vph\in C_c^1(\overline{Q_T})$, $\vph(T,\cdot) = 0$, $\vph \cdot \vc{n}|_{\Gamma_t} = 0$ such that  $\Div \vph = 0$ in $Q_T$  when passing to the limit in the momentum equation \eqref{m3}. Relations \eqref{conv_001}--\eqref{conv_018} imply
$$
\int_0^t \int_{\Omega_t} \overline \varrho \vU \cdot \partial_t \vph + \overline{\varrho \vU \otimes \vU}:\Grad \vph + \mu(\vtb)\nabla_x\vU:\Grad \vph - \varrho^{(1)} \nabla_x F \cdot \vph \dxdt 
= \int_{\Omega_0}\overline \varrho \vU_0 \cdot \vph(0,\cdot)\dx
$$
where $\overline{\varrho \vU\otimes \vU}$ is the weak limit of $\varrho_\varepsilon \vu_\varepsilon \otimes \vu_\varepsilon$. Note that if 
\bFormula{eq:goal0}
\int_0^T\int_{\Omega_t}\overline{\varrho \vU \otimes \vU}:\Grad \vph\dx = \int_0^T\int_{\Omega_t}(\overline \varrho \vU \otimes \vU):\Grad \vph \dx
\eF
for all $\vph\in C_c^1(\overline{Q_T})$, $\vph(T,\cdot) = 0$, $\vph \cdot \vc{n}|_{\Gamma_t} = 0$ such that  $\Div \vph = 0$ in $Q_T$,
the Theorem~\ref{t:main} is proven. The rest of this paper is devoted to the analysis of this particular limit.

We start with a simple observation, that due to the uniform bound of $\vre\vue\otimes\vue$ in $L^2(0,T;L^{\frac{30}{29}}(\R^3))$ it suffices to show \eqref{eq:goal0} for test functions compactly supported in $\overline Q_T \cap \left((0,T)\times \mathbb R^3\right)$.

\section{Limit in the convective term}\label{s:convec}

\subsection{Helmholtz decomposition}\label{s:helmholtz}

We introduce the Helmholtz decomposition $\vv = \H_t[\vv] + \H_t^{\bot}[\vv]$ in $L^2(\Omega_t;\mathbb{R}^3)$ in a standard way. Namely we define the projection $\H_t^{\bot}[\vv] = \Grad\Psi$ as the unique solution to the Neumann problem
\begin{equation}\label{eq:df_Psi}
\Delta\Psi = \Div \vv \text{ in } \Omega_t, \quad \Grad\Psi\cdot\vn = \vv\cdot\vn \text{ on } \Gamma_t, \quad \int_{\Omega_t} \Psi \,\dx = 0.
\end{equation}
Hence, it is easy to observe that 
\begin{equation*}
\Div \H_t[\vv] = 0 \,\, \text{ in } \Omega_t \quad \text{ and } \quad 
\H_t[\vv]\cdot\vn = 0\,\,\text{ on } \Gamma_t.
\end{equation*}

\begin{Lemma}
Let $\vz$ be such that $\partial_t \Div \vz \in W^{2,2}(\Omega_t)$ and $\partial_t \vz \in W^{3/2,2}(\Gamma_t)$. Then $\partial_t \H_t[\vz] \in W^{1,2}(\Omega_t)$.
\end{Lemma}
\begin{proof}
See \cite{sozo} and \cite[Section 3.1]{FKNNS14}.
\end{proof}

\subsection{Compactness of the solenoidal part}

Although $\Div \vU = 0$ a.e. in $Q_T$, we cannot conclude that $\Ht[\vU] = 0$ due to the inhomogeneous boundary condition $\vU\cdot\vn = \vV\cdot\vn$ on $\Gamma_t$. Instead we have 
\begin{equation*}
\vU = \Ht[\vU] + \nabla_x W, \qquad \text{where } \nabla_x W = \Hpt[\vV].
\end{equation*}
Note that 
\begin{equation*}
\intint{\Ht[\vv]\cdot \vph} = \intint{\vv\cdot\Ht[\vph]}
\end{equation*}
for all $\vv,\vph \in L^2(Q_T)$, this property will be used extensively throughout the rest of this paper.

Moreover, we also have
\begin{equation*}
\|\Ht[\vv]\|_{L^q(\Omega_\tau)} \leq c(q) \|\vv \|_{L^q(\Omega_\tau)}
\end{equation*}
for any $2 \leq q < \infty$ and $\tau \in [0,T]$ due to the elliptic regularity theory. Since the domains $\Omega_\tau$ are regular, the constant $c(q)$ can be chosen independently of $\tau$, see \cite[Theorem 1.2]{FKS}.

Convergences  from Section \ref{sec:uni:est} imply that for $\ep \to 0$
\begin{equation*}
\begin{split}
\Ht[\vue] &\weak \Ht[\vU] \quad \text{ weakly in } L^2(Q_T), \\
\nabla_x \Ht[\vue] &\weak \nabla_x\Ht[\vU] \quad \text{ weakly in } L^2(Q_T), \\
\Hpt[\vue] &\weak \nabla_x W \quad \text{ weakly in } L^2(Q_T), \\
\nabla_x \Hpt[\vue] &\weak \nabla^2_x W \quad \text{ weakly in } L^2(Q_T).
\end{split}
\end{equation*}

We want to prove strong convergence of $\Ht[\vue]$ to $\Ht[\vU]$ in $L^2(Q_T)$. To this end we test the momentum equation \eqref{m3} by $\Ht[\vph]$ with $\vph \in C^1_c(\overline{Q_T})$, $\vph(T) = 0$, $\vph\cdot\vn = 0$ on $\Gamma_\tau$.

We denote
\begin{equation*}
I_\vph^\ep(t) := \int_{\Omega_t} (\vre\vue)(t,x)\cdot\Ht[\vph(t,x)] \dx = \int_{\Omega_t} \Ht[(\vre\vue)(t,x)]\cdot\vph(t,x)\dx 
\end{equation*}
and, consequently,
\begin{align} \nonumber
I_\vph^\ep(t) - I_\vph^\ep(t') &= \int_t^{t'}\int_{\Omega_\tau} \left(\vre\vue\otimes\vue - \tn{S}(\vte,\nabla_x\vue)\right):\nabla_x\Ht[\vph] + \frac{\vr_\ep-\vrb}{\ep}\nabla_x F \cdot \Ht[\vph] \dxdt \\ 
 &+ \int_t^{t'}\int_{\Omega_\tau} \vre\vue\cdot\de_t\Ht[\vph] \dxdt. \label{eq:AA1}
\end{align}
We estimate the right hand side of \eqref{eq:AA1} by 
\eqref{est_2}, \eqref{est_4}, \eqref{est_5}, \eqref{est_7}
and the regularity of the time derivative of $\Ht[\vph]$. We end up with 
\begin{equation*}
\left|I_\vph^\ep(t) - I_\vph^\ep(t')\right| \leq C |t-t'|^{1/2}.
\end{equation*}
Fix a time interval $[T_1,T_2]$ and an open set $K \subset \R^3$ such that $[T_1,T_2]\times \overline{K} \subset Q_T$. We use the Arzel\'a-Ascoli theorem to conclude that $I_\vph^\ep$ is precompact in $C(T_1,T_2)$ and therefore 
\begin{equation*}
\Ht[\vre\vue] \sil \Ht[\vrb\vU] \quad \text{ strongly in } C_{w}(T_1,T_2;L^{5/4}(K)).
\end{equation*}
This implies that
\begin{equation*}
\Ht[\vre\vue] \sil \Ht[\vrb\vU] \quad \text{ strongly in } L^p(T_1,T_2;W^{-1,2}(K)), 1\leq p < \infty.
\end{equation*}
We also have
\begin{equation*}
(\vre-\vrb)\vue = \ep\frac{\vre-\vrb}{\ep}\vue \sil 0 \mbox{ strongly in }L^2(T_1,T_2;L^{30/23}(K))
\end{equation*}
 which yields the same property for $\Ht[(\vre-\vrb)\vue]$ and for $\Hpt[(\vre-\vrb)\vue]$. Therefore we can write
\begin{equation*}
\vrb\Ht[\vue] \vue
= (\Ht[(\vrb-\vre)\vue]+ \Ht[\vre\vue])\cdot\vue  \weak \vrb |\vU|^2 \text{ weakly in } L^1((T_1,T_2) \times K);
\end{equation*}
in particular
\begin{equation*}
\int_{T_1}^{T_2}\int_{K}|\Ht[\vue]|^2  = \int_{T_1}^{T_2}\int_{K} {\Ht[\vue]\cdot \vue}  \to \int_{T_1}^{T_2}\int_{K} |\vU|^2 
\end{equation*}
and we conclude that 
\begin{equation*}
\Ht[\vue] \sil \Ht[\vU] \quad \text{ strongly in } L^2([T_1,T_2] \times \overline K).
\end{equation*}
We deduce that in order to show \eqref{eq:goal0} it is 
enough to prove that
\begin{equation}\label{eq:goal}
\intint{\Ht^\perp[\vre(\vue-\vV)]\otimes\Ht^\perp[(\vue-\vV)]:\nabla_x \vph} \sil 0
\end{equation}
for all test functions $\vph \in C^1_c(\overline{Q_T})$, $\Div \vph = 0$, $\vph(0,\cdot) = \vph(T,\cdot) = 0$, $\vph\cdot\vn|_{\Gamma_t} = 0$  (for details see \cite[Section 3.3]{FKNNS14}).

\subsection{Acoustic equation}

We rewrite the continuity equation \eqref{m2} (with $B \equiv 1, b \equiv 0$) and the momentum equation \eqref{m3} 
in the form of the acoustic analogy. To this end we reformulate both in a new variables
\begin{equation}\label{eq:newvar}
\vr_\ep^{(1)} = \frac{\vr_\ep-\vrb}{\ep}, \qquad \vz_\ep = \vr_\ep(\vu_\ep - \vV).
\end{equation}
The continuity equation \eqref{m2} reads as 
\begin{equation}\label{eq:CEac}
\intint{\ep \vr_\ep^{(1)} \partial_t \varphi + \vz_\ep\cdot\nabla_x\varphi} = -\intint{\ep \vr_\ep^{(1)} \vV\cdot\nabla_x\varphi}
\end{equation}
for any $\varphi \in C^\infty_c(\overline{Q_T})$, $\varphi(0,\cdot)=0$, $\varphi(T,\cdot) = 0$ and the momentum equation \eqref{m3} can be written in the following form
\begin{align}
&\intint{\ep\vz_\ep\cdot\de_t\vph + \left[\frac{[p(\vre,\vte)]_{ess} - p(\vrb,\vtb)}{\ep} - \vrb F\right]\Div\vph} \nonumber \\
& = \ep \intint{\frac{\vrb-\vre}{\ep}\nabla_x F \cdot \vph} + \ep\intint{\left(\HH^1_\ep:\nabla_x\vph + \vh^2_\ep\cdot\vph\right)} \label{eq:MEac}
\end{align}
for any $\vph \in C^\infty_c(\overline{Q_T})$, $\vph(0,\cdot)=0$, $\vph(T,\cdot)=0$, $\vph\cdot\vn|_{\Gamma_t} = 0$ for $t\in [0,T]$, where we set
\begin{align}
\HH^1_\ep &= -\vre\vue\otimes\vue + \tn{S}_\ep - \frac{[p(\vre,\vte)]_{res}}{\ep^2}\tn{I} + \vre\vue\otimes\vV ,\label{eq:H1}\\
\vh^2_\ep &= \vre \de_t \vV + \vre\vue\cdot\nabla_x\vV \label{eq:h2}.
\end{align}
We also need the entropy balance \eqref{m5} in the form
\begin{equation}\label{eq:EBac}
\intint{\ep\left(\vre\frac{s(\vre,\vte) - s(\vrb,\vtb)}{\ep}\right)\de_t \varphi} = \intint{\vh^3_\ep\cdot\nabla_x\varphi} - \langle\sigma_\ep;\varphi\rangle
\end{equation}
for all  $\varphi \in C^\infty_c(\overline{Q_T})$, $\varphi(0,\cdot)=0$, $\varphi(T,\cdot) = 0$, where 
\begin{equation}\label{eq:h3}
\vh^3_\ep = \frac{\kappa(\vte)}{\vte}\nabla_x\vte - \vre(s(\vre,\vte)-s(\vrb,\vtb))\vue .
\end{equation}

In the next step we rewrite the system \eqref{eq:CEac}, \eqref{eq:MEac}, \eqref{eq:EBac} using new set of variables, namely we define
\begin{equation}\label{eq:rdef}
r_\ep = \vr_\ep^{(1)} + \frac{A}{\zeta}\vre\frac{s(\vre,\vte) - s(\vrb,\vtb)}{\ep} - \frac{1}{\zeta}\vrb F,
\end{equation}
where we set
\begin{equation}\label{eq:Aom}
A = \frac{\de_\vt p(\vrb,\vtb)}{\vrb \de_\vt s(\vrb,\vtb)}, \qquad \zeta = \de_\vr p(\vrb,\vtb) + A\frac{\de_\vt p(\vrb,\vtb)}{\vrb}.
\end{equation}
Theorem \ref{t:transport} yields
\begin{equation}\label{eq:help}
\intint{F\de_t\varphi} = -\intint{\Div(F\varphi\vV)} = -\intint{\vV\cdot\nabla_x F \varphi + F\vV\cdot\nabla_x\varphi}.
\end{equation}
Hence, we sum \eqref{eq:CEac} and an appropriate multiple of \eqref{eq:EBac} and by \eqref{eq:help} we end up with
\begin{equation}\label{eq:CEac2}
\intint{\ep r_\ep \partial_t \varphi + \vz_\ep\cdot\nabla_x\varphi} = \ep\intint{\left(\vh^4_\ep\cdot\nabla_x\varphi + h^5_\ep \varphi\right)} - 
\frac{A}{\zeta}\langle\sigma_\ep;\varphi\rangle
\end{equation}
for all $\varphi \in C^\infty_c(\overline{Q_T})$, $\varphi(0,\cdot)=0$, $\varphi(T,\cdot) = 0$, where
\begin{align}
\vh^4_\ep &= -\vre^{(1)}\vV + \frac{1}{\ep}\frac{A}{\zeta}\vh^3_\ep + \frac{\vrb}{\zeta}F\vV ,
\label{eq:h4}\\
h^5_\ep &= \frac{\vrb}{\zeta}\vV\cdot\nabla_x F \label{eq:h5}.
\end{align}
The acoustic version of the momentum equation \eqref{eq:MEac} is rewritten as
\begin{align}\label{eq:MEac2}
&\intint{\ep\vz_\ep\cdot\partial_t\vph + \zeta r_\ep\Div\vph} \nonumber \\
& =  \ep \intint{\frac{\vrb-\vre}{\ep}\nabla_x F \cdot \vph} + \ep\intint{\left(\HH^1_\ep:\nabla_x\vph + \vh^2_\ep\cdot\vph + h^6_\ep \Div \vph\right)}
\end{align}
for any  $\vph \in C^\infty_c(\overline{Q_T})$, $\vph(0,\cdot)=0$, $\vph(T,\cdot)=0$, $\vph\cdot\vn|_{\Gamma_t} = 0$ for $t\in [0,T]$, where we use the definitions of $A$ and $\zeta$ stated in \eqref{eq:Aom} together with the notation
\begin{align}
\frac{[p(\vre,\vte)]_{ess} - p(\vrb,\vtb)}{\ep} &= \partial_\vr p(\vrb,\vtb) \vr_\ep^{(1)} + \partial_\vt p(\vrb,\vtb) \vt_\ep^{(1)} + h^7_\ep ,\nonumber \\
\left[\vre\frac{s(\vre,\vte) - s(\vrb,\vtb)}{\ep}\right]_{ess} &= \vrb\partial_\vr s(\vrb,\vtb) \vr_\ep^{(1)} + \vrb\partial_\vt s(\vrb,\vtb)\vt_\ep^{(1)} + h^8_\ep , \nonumber \\
h^6_\ep &= \frac{1}{\ep}\left(A\left[\vre\frac{s(\vre,\vte) - s(\vrb,\vtb)}{\ep}\right]_{res} + Ah^8_{\ep} - h^7_\ep\right) , \label{eq:h6}
\end{align}
and since both $p$ and $s$ are twice continuously differentiable we have by \cite[Proposition 5.2]{FeNo6}
\begin{align}
\text{ess} \sup_{t\in (0,T)} \int_{\Omega_t} \abs{h^7_\ep(t,x)} \dx \leq C\ep , \label{eq:h7} \\
\text{ess} \sup_{t\in (0,T)} \int_{\Omega_t} \abs{h^8_\ep(t,x)} \dx \leq C\ep. \label{eq:h8}
\end{align}

\subsection{Time lifting}


The right hand side of \eqref{eq:CEac2} contains a measure $\sigma_\ep$ and thus the associated solution is not necessarily continuous. To prevent this, we adopt the method described in \cite[Section 5.4.7]{FeNo6} - called time lifting - and we introduce a "primitive" measure $\Sigma_\varepsilon$ on $Q_t$  defined as
\begin{equation*}
\langle \Sigma_\varepsilon; \varphi\rangle = \langle \sigma_\varepsilon; I[\varphi]\rangle,
\end{equation*}
with 
\begin{equation*}
I[\varphi](t,x) = \int_0^t \varphi(\tau, \widetilde{\vc{X}}(\tau,x) ){\rm d}\tau
\end{equation*}
where $\widetilde{\vc{X}}(\tau,x)$ is a solution to
\begin{equation*}
\frac{{\rm d}}{{\rm d}t} \widetilde{\vc{X}}(\tau, x) = \vc{V} \Big( \tau, \widetilde{\vc{X}}(\tau, x) \Big)
\end{equation*}
with the condition $\widetilde{\vc{X}}(t,x) = x$. It follows that
\begin{equation*}
\langle \Sigma_\varepsilon; \partial_t \varphi + \Grad \varphi \cdot {\bf V}\rangle = \langle \sigma_\ep; \varphi\rangle
\end{equation*}
and $\Sigma_\varepsilon$ can be also identified as a mapping $[0,T)\mapsto \mathcal M^+(\Omega_t)$, where $\Sigma_\varepsilon(t)$ is defined by the duality with a function $\widetilde{\varphi} \in C(\overline{\Omega_t})$ as follows
\begin{equation*}
\langle \Sigma_\varepsilon(t); \widetilde{\varphi}\rangle = \lim_{\delta\to 0_+}\langle \sigma_\varepsilon; \psi_\delta \widetilde{\varphi}_{ext}\rangle,\ \mbox{for almost all }t\in [0,T)
\end{equation*}
with 
\begin{equation*}
\psi_\delta(\tau) = \left\{\begin{array}{l} 0\ \mbox{for}\ \tau\leq t\\ \frac1\delta(\tau-t)\ \mbox{for}\ \tau\in(t, t + \delta)\\ 1 \ \mbox{for}\ \tau\geq t + \delta\end{array}
\right.
\end{equation*}
and $\widetilde{\varphi}_{ext}(\tau,x)$ is the extension of the function $\widetilde{\varphi}(x)$ given as
\begin{equation*}
\widetilde{\varphi}_{ext}(\tau,\widetilde{\vc{X}}(\tau,x)) = \widetilde{\varphi}(x).
\end{equation*}
It holds that
\begin{equation}\label{est.Sigma}
\mbox{ess sup}_{t\in (0,T)} \|\Sigma_\varepsilon(t)\|_{\mathcal M^+(\overline{\Omega_\tau})}\leq \|\sigma_\varepsilon\|_{\mathcal M^+(\overline{Q_T})}\leq \varepsilon^2 c.
\end{equation}
We use a notation
\begin{equation*}
\langle \Sigma_\varepsilon(t); \varphi\rangle : = \int_{\Omega_t} \Sigma_\varepsilon(t) \varphi\dx .
\end{equation*}
We define a new variable
\begin{equation}\label{eq:Zdef2}
Z_\varepsilon = r_\varepsilon +  \frac A {\ep\zeta} \Sigma_\varepsilon.
\end{equation}
and we rewrite \eqref{eq:CEac2} and \eqref{eq:MEac2} as
\begin{equation}\label{eq:CEac21}
\intint{\ep Z_\ep \partial_t \varphi + \vz_\ep\cdot\nabla_x\varphi} = \ep\intint{\left(\vh^4_\ep\cdot\nabla_x\varphi + h^5_\ep \varphi  - \frac A{\ep \zeta} \Sigma_\ep {\bf V} \cdot \Grad \varphi \right)}
\end{equation}
and
\begin{align}\label{eq:MEac21}
\intint{\ep\vz_\ep\cdot\partial_t\vph &+ \zeta Z_\ep\Div\vph} =  \ep \intint{\frac{\vrb-\vre}{\ep}\nabla_x F \cdot \vph}\nonumber\\
 &+ \ep\intint{\left(\HH^1_\ep:\nabla_x\vph + \vh^2_\ep\cdot\vph + \left(h^6_\ep +\frac A{\ep^2} \Sigma_\ep\right)\Div \vph\right)}.
\end{align}

\subsection{Reduction to a finite number of modes in moving domain}

We want to reduce our problem and study it only after projecting it into finite number of modes. To this end we follow the strategy developed in \cite[Section 4]{FKNNS14} and we introduce the eigenvalue problem 
\begin{equation*}
\nabla_x\omega =
-\lambda(t)\va, \quad \Div \va = -\lambda(t)\om \quad \text{ in }
\Omega_t, \quad \va\cdot\vn = 0 \quad \text{ on } \Gamma_t,
\end{equation*}
which admits solutions in the form
\begin{equation*}
\va_j(t,x) =
\frac{\rm i}{\sqrt{\Lambda_j(t)}} \nabla_x \om_j(t,x), \quad
\lambda_j(t) = {\rm i}\sqrt{\Lambda_j(t)}, \quad j = 1,2,\dots .
\end{equation*}
Here $\Lambda_j(t)$, $\omega_j(t,\cdot)$ are eigenvalues and eigenfunctions of the Neumann Laplace problem
\begin{equation*}
-\Delta_x\omega = \Lambda(t)\omega \text{ in } \Omega_t, \quad \nabla_x\omega\cdot\vn = 0 \quad \text{ on } \Gamma_t,
\end{equation*}
which admits real eigenvalues
\begin{equation*}
0 = \Lambda_0(t) < \Lambda_1(t) \leq \Lambda_2(t) \leq ...
\end{equation*}
We have that 
$\{\va_j(t,\cdot)\}_{j=1}^\infty$ forms an orthonormal basis in  $\Hpt(L^2(\Omega_t))=\overline{\left\{\text{span}\left\{{\rm i} \va_j(t,\cdot)\right\}_{j=1}^\infty\right\}}^{L^2(\Omega_t)}$ and the eigenspace of $\lambda_0(t) = 0$ is $\Ht(L^2(\Omega_t))$.

Since $\vV$ is smooth enough we use \cite[Theorem 4.3]{BLL} to conclude 
\begin{equation*}
\abs{\frac{1}{\Lambda_j(t_1)} - \frac{1}{\Lambda_j(t_2)}} \leq C \abs{t_1-t_2},
\end{equation*}
for $t_1,t_2 \in [0,T]$, however no such property holds for the eigenfunctions $\omega_j$ and $\va_j$. Therefore we work with the projections on the eigenspaces spanned by finite number of eigenfunctions. More precisely, fixing an integer $M > 0$ we define projections
\begin{align}
P_M[\varphi](t,\cdot) &:= \sum_{j=1}^M \omega_j(t,\cdot)\int_{\Omega_t}\varphi(t,y)\omega_j(t,y)\dy, \quad \varphi\in L^2(Q_T) , \nonumber \\
\vQ_M[\vph](t,\cdot) &:= \sum_{j=1}^M \va_j(t,\cdot)\int_{\Omega_t}\vph(t,y)\cdot\va_j(t,y)\dy, \quad \vph\in L^2(Q_T). \nonumber
\end{align}

As explained in \cite[Section 4]{FKNNS14} the Lipschitz continuity of projections $P_M$, $\vQ_M$ cannot be expected in general on the whole time interval $[0,T]$ even if $\vV$ is smooth, because such property holds only under the assumption 
\begin{equation}\label{eq:MM1}
\Lambda_{M+1} \neq \Lambda_M.
\end{equation}
It may happen that there is no $M > 0$ such that \eqref{eq:MM1} holds for all $t \in [0,T]$. This is solved by introducing a finite cover of $[0,T]$ formed by intervals $\{I_l\}_{l=1}^n$, where for every $l \in \{1,...,n\}$ there exists $M_l > M$ for some fixed $M$ such that 
\begin{equation}\label{eq:MM2}
\Lambda_{M_l+1} \neq \Lambda_{M_l} \qquad \text{for all } t \in I_l.
\end{equation}

We take $\psi \in C^\infty_c(I_l)$ and $\varphi \in C^\infty_c((0,T)\times\Rtri)$ and use $\psi(t)P_{M_l}[\varphi](t,x)$ as a test function in \eqref{eq:CEac21} to obtain
\begin{align}\label{eq:CEac3}
&\int_{I_l} \psi \int_{\Omega_t} \left(\ep \partial_tP_{M_l}[Z_\ep] + \Div \vQ_{M_l}[\vz_\ep]\right)\varphi \dxdt \\ \nonumber 
= - &\ep\int_{I_l} \psi \int_{\Omega_t}\left( \vh^4_\ep\cdot\nabla_xP_{M_l}[\varphi] + h^5_\ep P_{M_l}[\varphi]  - 
\frac{A}{\ep \zeta}\Sigma_\ep \vV\cdot \Grad P_{M_l}[\varphi]\right)\dxdt \\ \nonumber
+ &\ep\int_{I_l} \psi \int_{\Omega_t}\left( Z_\ep \partial_t P_{M_l}[\varphi] - \vV \cdot \nabla_x \left(P_{M_l}[Z_\ep]\varphi\right) - P_{M_l}[Z_\ep]\partial_t \varphi\right) \dxdt.
\end{align}

We take $\vph \in C^\infty_c((0,T)\times\Rtri)$, $\vph\cdot\vn = 0$ on $\Gamma_t$, and use $\psi(t)\vQ_{M_l}[\vph](t,x)$ as a test function in \eqref{eq:MEac21} to obtain
\begin{align}\label{eq:MEac3}
&\int_{I_l} \psi \int_{\Omega_t} \left(\ep \partial_t \vQ_{M_l}[\vz_\ep] + \zeta \nabla_x P_{M_l}[Z_\ep]\right)\cdot\vph \dxdt \\ \nonumber 
= &- \ep \int_{I_l} \psi \int_{\Omega_t}\left( \frac{\vrb-\vre}{\ep}\nabla_x F \cdot \vQ_{M_l}[\vph]\right.\\\nonumber&\left. +\, \HH^1_\ep:\nabla_x\vQ_{M_l}[\vph] + \vh^2_\ep\cdot\vQ_{M_l}[\vph] + \left(h^6_\ep  +\frac A{\ep^2} \Sigma_\ep\right)\Div \vQ_{M_l}[\vph]\right) \dxdt  \\ \nonumber
+ &\ep \int_{I_l} \psi \int_{\Omega_t}\left( \vz_\ep\cdot \partial_t \vQ_{M_l}[\vph] - \vV \cdot \nabla_x \left(\vQ_{M_l}[\vz_\ep]\cdot\vph\right) - \vQ_{M_l}[\vz_\ep]\cdot\partial_t \vph\right) \dxdt.
\end{align}

We observe that introducing $d_{\ep,l} := P_{M_l}[Z_\ep]$ and $\nabla_x \Psi_{\ep,l} := \vQ_{M_l}[\H^\perp[\vz_\ep]]$, the system of equations \eqref{eq:CEac3}-\eqref{eq:MEac3} can be formally written as 
\begin{align}\label{eq:ACel1}
\ep \partial_t \del + \Delta_x \psel = \ep \felj \\
\ep \de_t \nabla_x \psel + \zeta \nabla_x \del = \ep \feld\label{eq:ACel2}
\end{align}
for some $\felj, \feld$, which is to be satisfied in $\{(t,x): t \in I_l, x \in \Omega_t\}$. However, we will  rather work with the weak formulation which reads as
\begin{align}\label{eq:CEac4}
&\int_{I_l} \psi \int_{\Omega_t} \left(\ep \partial_t \del + \Delta_x \psel\right)\varphi \dxdt \\ \nonumber 
= - &\ep\int_{I_l} \psi \int_{\Omega_t}\left( \vh^4_\ep\cdot\nabla_xP_{M_l}[\varphi] + h^5_\ep P_{M_l}[\varphi] -
\frac{A}{\ep \zeta} \Sigma_\ep\vV\cdot \Grad P_{M_l}[\varphi]\right) \dxdt \\ \nonumber
+ &\ep \int_{I_l} \psi \int_{\Omega_t}\left( Z_\ep \partial_t P_{M_l}[\varphi] - \vV \cdot \nabla_x \left(\del\varphi\right) - \del\partial_t \varphi\right) \dxdt =: \ep\felj[\psi,\varphi]
\end{align}
for all $\psi \in C^\infty_c(I_l)$ and $\varphi \in C^\infty_c((0,T)\times\Rtri)$ and
\begin{align}\label{eq:MEac4}
&\int_{I_l} \psi \int_{\Omega_t} \left(\ep \partial_t \nabla_x\psel + \zeta \nabla_x \del\right)\cdot\vph \dxdt \\ \nonumber 
= - &\ep \int_{I_l} \psi \int_{\Omega_t}\left( \frac{\vrb-\vre}{\ep}\nabla_x F \cdot \vQ_{M_l}[\vph] + \HH^1_\ep:\nabla_x\vQ_{M_l}[\vph] \right.\\&\left.\nonumber+ \vh^2_\ep\cdot\vQ_{M_l}[\vph] + \left(h^6_\ep  +\frac A{\ep^2} \Sigma_\ep\right)\Div \vQ_{M_l}[\vph]\right) \dxdt  \\ \nonumber
+ &\ep \int_{I_l} \psi \int_{\Omega_t}\left( \vz_\ep\cdot \partial_t \vQ_{M_l}[\vph] - \vV \cdot \nabla_x \left(\nabla_x \psel\cdot\vph\right) - \nabla_x\psel\cdot\partial_t \vph\right) \dxdt =: \ep\feld[\psi,\vph].
\end{align}
for all $\psi \in C^\infty_c(I_l)$ and $\vph \in C^\infty_c((0,T)\times\Rtri)$, $\vph\cdot\vn = 0$.


\subsection{Conclusion of the proof}

Let us remind that we want to prove \eqref{eq:goal}. For this order we  introduce the following partition of unity on a time interval $[0,T]$ 
\begin{equation*}
\sum_{l = 1}^n \psi_l(t) = 1  \quad \text{ for all } t \in [0,T], \quad\mbox{ where }
\psi_l \in C^\infty_c(I_l)  \quad l = 1,...,n,
\end{equation*}
where $I_l$ are intervals introduced in the previous section such that \eqref{eq:MM2} holds and we write 
\begin{equation*}
\begin{split}
&\intint{\Hpt[\vre(\vue-\vV)]\otimes\Hpt[(\vue-\vV)]:\nabla_x \vph} \\ 
& \qquad = \sum_{l=1}^n \int_{I_l}\psi_l \int_{\Omega_t} \Hpt[\vre(\vue-\vV)]\otimes\Hpt[(\vue-\vV)] : \nabla_x \vph \dxdt
\end{split}
\end{equation*}
for test functions $\vph \in C^1_c(\overline{Q_T})$, $\Div \vph = 0$, $\vph(0,\cdot) = \vph(T,\cdot) = 0$, $\vph\cdot\vn|_{\Gamma_t} = 0$.

We split both terms of the product into the finite mode part and the remainder part
\begin{equation*}
\begin{split}
\Hpt[\vze]\otimes\Hpt[(\vue-\vV)] &= (\vQ_{M_l}[\Hpt[\vze]] + (\Hpt[\vze] - \vQ_{M_l}[\Hpt[\vze]])) \\  &\otimes (\vQ_{M_l}[\Hpt[\vue-\vV]] + (\Hpt[\vue-\vV] - \vQ_{M_l}[\Hpt[\vue-\vV]])).
\end{split}
\end{equation*}
Moreover, we also have
\begin{equation*}
\begin{split}
\Hpt[\vze] - \vQ_{M_l}[\Hpt[\vze]] &= \Hpt[(\vre-\vrb)(\vue-\vV)] - \vQ_{M_l}[\Hpt[(\vre-\vrb)(\vue-\vV)]]	 \\ \nonumber &+ \vrb (\Hpt[\vue-\vV] - \vQ_{M_l}[\Hpt[\vue-\vV]])
\end{split}
\end{equation*}
and we recall that for $\ep \to 0$
\begin{equation*}
(\vre-\vrb)\vue \sil 0 \quad \text{ strongly in } L^2(0,T,L^{30/23}(\Omega_t)),
\end{equation*}
so the same holds also for $\Hpt[(\vre-\vrb)(\vue-\vV)] - \vQ_{M_l}[\Hpt[(\vre-\vrb)(\vue-\vV)]]$.

We want to show that the remainder terms are small if we choose $M_l$ large enough. We start with a useful expression for the $L^2$-norm of $\Div \vue$
\begin{equation*}
\begin{split}
\|\Div \vue \|_{L^2(\Omega_t)}^2 = \|\Div (\vue - \vV) \|_{L^2(\Omega_t)}^2 = \left\| \sum_{j=1}^\infty \Div \va_j \int_{\Omega_t} (\vue-\vV)\cdot\va_j \dx \right\|_{L^2(\Omega_t)}^2 \\
 = \left\| \sum_{j=1}^\infty \sqrt{\Lambda_j}\omega_j \int_{\Omega_t} (\vue-\vV)\cdot\va_j \dx \right\|_{L^2(\Omega_t)}^2 =  \sum_{j=1}^\infty \Lambda_j \left(\int_{\Omega_t} (\vue-\vV)\cdot\va_j \dx \right)^2
\end{split}
\end{equation*}
and we use it as follows
\begin{equation*}
\begin{split}
\left\| \Hpt[\vue-\vV] - \vQ_{M_l}[\Hpt[\vue-\vV]] \right\|_{L^2(\Omega_t)}^2 = \sum_{j > M_l} \left(\int_{\Omega_t} (\vue-\vV)\cdot\va_j \dx \right)^2 \\ 
\leq \frac{1}{\inf_{j>M_l} \Lambda_j(t)}\|\Div \vue \|_{L^2(\Omega_t)}^2 \leq \frac{1}{\inf_{j>M} \Lambda_j(t)}\|\Div \vue \|_{L^2(\Omega_t)}^2.
\end{split}
\end{equation*}
We observe that the quantity
\begin{equation*}
\frac{1}{\inf_{t\in[0,T], j>M} \Lambda_j(t)}
\end{equation*}
can be made as small as we want by the choice of $M$.

We are left with the goal of estimating the product of finite modes terms, namely
\begin{equation*}
\int_{I_l} \psi_l \int_{\Omega_t} \vQ_{M_l}[\Hpt[\vze]] \otimes \vQ_{M_l}[\Hpt[(\vue-\vV)]] : \nabla_x \vph \dxdt \sil 0 \quad \mbox{ as }\ep \to 0
\end{equation*}
for all test functions $\vph \in C^1_c(\overline{Q_T})$, $\Div \vph = 0$, $\vph(0,\cdot) = \vph(T,\cdot) = 0$, $\vph\cdot\vn|_{\Gamma_t} = 0$ and for $l = 1,...,n$, which can be rewritten equivalently to 
\begin{equation*}
\int_{I_l} \psi_l \int_{\Omega_t} \vQ_{M_l}[\Hpt[\vze]] \otimes \vQ_{M_l}[\Hpt[\vze]] : \nabla_x \vph \dxdt \sil 0 \quad 
\mbox{ as }\ep \to 0.
\end{equation*}

We recall that we denoted $\nabla_x \psel = \vQ_{M_l}[\Hpt[\vze]]$ and we have equations \eqref{eq:ACel1}-\eqref{eq:ACel2}, or more precisely their weak formulations \eqref{eq:CEac4}-\eqref{eq:MEac4} at our disposal. Integrating by parts we have
\begin{equation*}
\begin{split}
&\int_{I_l} \psi_l \int_{\Omega_t} (\nabla_x \psel \otimes \nabla_x\psel) : \nabla_x \vph \dxdt  \\
& \quad = - \int_{I_l} \psi_l \int_{\Omega_t} \Delta_x \psel \nabla_x\psel \cdot \vph \dxdt - \frac 12 \int_{I_l} \psi_l \int_{\Omega_t} \nabla_x | \nabla_x\psel |^2 \cdot \vph \dxdt \\ 
& \quad = - \int_{I_l} \psi_l \int_{\Omega_t} \Delta_x \psel \nabla_x\psel \cdot \vph \dxdt
\end{split}
\end{equation*}
where we used that the second term on the middle line is zero due to the fact that $\Div \vph = 0$. We use the equation \eqref{eq:CEac4} with $\varphi = \nabla_x\psel\cdot \vph$ and \eqref{eq:MEac4} with $\del \vph$ as a test function together with the transport theorem to obtain
\begin{align} \nonumber
& - \int_{I_l}\psi_l \int_{\Omega_t} \Delta_x \psel \nabla_x\psel \cdot \vph \dxdt = \ep \int_{I_l}\psi_l \int_{\Omega_t} \de_t \del \nabla_x\psel \cdot \vph \dxdt - \ep \felj[\psi_l,\nabla_x\psel \cdot \vph] \\ \nonumber
& \qquad = \ep \int_{I_l}\psi_l \int_{\Omega_t} \left(\de_t( \del \nabla_x\psel) - \del\de_t\nabla_x\psel\right)\cdot \vph \dxdt - \ep \felj[\psi_l,\nabla_x\psel \cdot \vph] \\ \nonumber
& \qquad = \ep \int_{I_l} \int_{\Omega_t} \left(\de_t( \psi_l\del \nabla_x\psel\cdot\vph ) - \del\nabla_x\psel\cdot \de_t(\psi_l \vph)\right) \dxdt \\ \nonumber
& \qquad + \int_{I_l}\psi_l \int_{\Omega_t} \zeta \nabla_x \frac{\del^2}{2} \cdot \vph \dxdt - \ep \feld[\psi_l,\del\vph] - \ep \felj[\psi_l,\nabla_x\psel \cdot \vph] \\ \nonumber
& \qquad = -\ep \int_{I_l} \psi_l \int_{\Omega_t} \vV \cdot\nabla_x (\del \nabla_x\psel\cdot\vph )\dxdt - \ep \int_{I_l} \int_{\Omega_t} \del\nabla_x\psel\cdot \de_t(\psi_l \vph) \dxdt \\ \label{eq:666}
& \qquad - \ep \feld[\psi_l,\del\vph] - \ep \felj[\psi_l,\nabla_x\psel \cdot \vph],
\end{align}
where the last equality is true because $\Div \vph = 0$. All the terms on the right hand side are multiplied by $\ep$, so to conclude our proof it is enough to show that the integrals contained in the right hand side are bounded independently of $\ep$ and $l$. 
The first two integrals on the right hand side of \eqref{eq:666} contain only smooth functions and therefore are obviously bounded, so we focus only on the terms $\felj$ and $\feld$. By \eqref{eq:CEac4} we have 
\begin{align}\label{eq:felj}
\felj[\psi,\varphi] =& - \int_{I_l} \psi \int_{\Omega_t}\left( \vh^4_\ep\cdot\nabla_xP_{M_l}[\varphi] + h^5_\ep P_{M_l}[\varphi] -
\frac{A}{\ep \zeta} \Sigma_\ep \vV\cdot \Grad P_{M_l}[\varphi]\right) \dxdt \\ \nonumber & + \int_{I_l} \psi \int_{\Omega_t} Z_\ep \de_t P_{M_l}[\varphi] \dxdt +  I_1[\psi,\varphi],
\end{align}
and one checks that $I_1[\psi_l,\nabla_x\psel \cdot \vph]$ contains again only smooth functions and therefore is bounded. Similarly 
\begin{align}\label{eq:feld}
\feld[\psi,\vph]  = &- \int_{I_l} \psi \int_{\Omega_t}\left( \frac{\vrb-\vre}{\ep}\nabla_x F \cdot \vQ_{M_l}[\vph] + \HH^1_\ep:\nabla_x\vQ_{M_l}[\vph] + \vh^2_\ep\cdot\vQ_{M_l}[\vph]\right) \dxdt\\ \nonumber 
& - \int_{I_l} \psi \int_{\Omega_t} \left(\left(h^6_\ep +\frac A{\ep^2} \Sigma_\ep\right)\Div \vQ_{M_l}[\vph] - \vze\cdot \de_t\vQ_{M_l}[\vph] \right) \dxdt + I_2[\psi,\varphi],
\end{align}
with $I_2[\psi_l,\del\vph]$ containing only smooth functions and therefore bounded. Recalling \eqref{eq:newvar}, \eqref{eq:H1}, \eqref{eq:h2}, \eqref{eq:h3}, \eqref{eq:rdef}, \eqref{eq:h4}, \eqref{eq:h5}, \eqref{eq:h6} and \eqref{eq:Zdef2} we have
\begin{equation*}
\begin{split}
\vze &= \vre(\vue-\vV) \\
Z_\ep &= \vre^{(1)} + \frac{A}{\zeta}\vre\frac{s(\vre,\vte) - s(\vrb,\vtb)}{\ep} - \frac{1}{\zeta} \vrb F + \frac{A}{\ep\zeta} \Sigma_\ep \\
\HH^1_\ep &= -\vre\vue\otimes\vue + \tn{S}_\ep - \frac{[p(\vre,\vte)]_{res}}{\ep^2}\tn{I} + \vre\vue\otimes\vV \\
\vh^2_\ep &= \vre \de_t \vV + \vre\vue\cdot\nabla_x\vV \\
\vh^4_\ep &= -\vre^{(1)}\vV + \frac{1}{\ep}\frac{A}{\zeta}\left(\frac{\kappa(\vte)}{\vte}\nabla_x\vte - \vre(s(\vre,\vte)-s(\vrb,\vtb))\vue\right) + \frac{\vrb}{\zeta}F\vV  \\ 
h^5_\ep &= \frac{\vrb}{\zeta}\vV\cdot\nabla_x F \\
h^6_\ep &= \frac{1}{\ep}\left(A\left[\vre\frac{s(\vre,\vte) - s(\vrb,\vtb)}{\ep}\right]_{res} + Ah^8_{\ep} - h^7_\ep\right) 
\end{split}
\end{equation*}
with $h^7_\ep, h^8_\ep$ satisfying \eqref{eq:h7}-\eqref{eq:h8}. By Lemma \ref{l:estimates} and \eqref{est.Sigma} we conclude that 
\begin{equation*}
\left|\felj[\psi_l,\nabla_x\psel \cdot \vph] \right| + \left|\feld[\psi_l,\del\vph] \right| \leq c.
\end{equation*}
The proof of Theorem \ref{t:main} is finished.


\end{document}